\documentclass[smallextended,envcountsect]{svjour3}
\smartqed
\usepackage{graphics,graphicx}
\usepackage{float}
\usepackage{amsmath, amssymb, enumerate, color, subfigure}
\newtheorem{thm}{Theorem}[section]
\newtheorem{prop}[thm]{Proposition}
\newtheorem{lem}[thm]{Lemma}
\newtheorem{rem}[thm]{Remark}

\journalname{JOTA}

\begin{document}

\title{Dividend and Capital Injection Optimization with Transaction Cost for Spectrally Negative L\'{e}vy Risk Processes}

\author{Wenyuan Wang \and Yuebao Wang \and  Xueyuan Wu }

\institute{Wenyuan Wang \at
             School of Mathematical Sciences, Xiamen University\\
						 Fujian 361005, China\\
              wwywang@xmu.edu.cn
           \and
           Yuebao Wang \at
					School of Mathematics, Soochow University\\
					Suzhou, 215006, China\\
					ybwang@suda.edu.cn
					\and Xueyuan Wu,  Corresponding Author  \at
              The University of Melbourne \\
              VIC 3010, Australia\\
              xueyuanw@unimelb.edu.au
}


\maketitle

\begin{abstract}
For an insurance company with reserve modelled by a spectrally negative L\'{e}vy process, we study the optimal \emph{impulse dividend and capital injection} (IDCI) strategy for maximizing the expected accumulated discounted net dividend payment subtracted by the accumulated discounted cost of injecting capital. In this setting, the beneficiary of the dividends injects capital to ensure a non-negative risk process so that the insurer never goes bankrupt. The optimal IDCI strategy together with its value function are obtained. Finally, a Brownian motion example is presented to illustrate the optimal IDCI strategy numerically.
\end{abstract}


\section{Introduction}\label{sec:1}

This paper aims to discuss the optimal \emph{impulse dividend and capital injection} (IDCI) strategy for an insurance company. The surplus level is modelled by a \emph{spectrally negative L\'{e}vy} (SNL) risk process, a widely used model in the literature. To enhance the practical relevance of the optimal dividend problem considered in this paper, we consider two real-life factors: the transaction costs on dividends and the capital injections. First, having the transaction costs included in the total cost of dividends is a natural addition; and second, allowing capital injections can protect the insurance company against the bankruptcy, thereby sustaining dividend payments in the long run. Through maximizing the expected accumulated discounted net dividend payments subtracted by the accumulated discounted cost of injecting capital under the proposed surplus process, we obtain the optimal IDCI strategy, which provides a useful reference for insurance companies when designing their long-term profit-sharing strategies.

The optimal dividend payout strategies have remained an active research field in the actuarial science literature for almost 60 years. Two survey papers, \cite{Albrecher09} and \cite{Avanzi09}, provide thorough and insightful reviews on the classical contributions and recent progress in the field. The earliest paper in the field, \cite{Finetti57}, proved that with the option to pay out dividends from its surplus to the beneficiary until the discrete time of ruin, an insurance company should adopt a barrier dividend strategy to maximize the expected total amount of discounted dividends until ruin. However, when dividends are imposed with fixed transaction costs, recent research findings in the literature suggest that the dividend optimization problem becomes an impulse (dividend) control problem and the optimal dividend strategy is an optimal impulse dividend (OID) strategy.

Research on OID strategies has attracted much attention for a decade and has progressed well under various surplus processes. In the classical \emph{Cram\'{e}r-Lundberg} (CL) risk model, \cite{Bai10a} studied an OID problem with transaction cost and tax on dividends as well as exponentially distributed claims. The obtained OID strategy \emph{reduces the reserve to level $u_{1}\in[0,u_{2})$ whenever it is above or equal to level $u_{2}$}, also called a $(u_{1},u_{2})$ strategy.
In the dual classical CL risk model, \cite{Zhou14} also considered an  OID problem with  fixed/proportional transaction cost on dividends and derived the OID strategy via a quasi-variational inequality argument. \cite{Bai10b} studied the OID problem with transaction costs on dividends for a  class of general diffusion risk processes and derived the $(u_{1},u_{2})$ OID strategy. In the context of SNL risk process, \cite{Loeffen09b} discussed an OID problem with transaction cost and showed that a $(u_{1},u_{2})$ strategy maximizes the expected accumulated present value of the net dividends. For the \emph{spectrally positive L\'{e}vy} (SPL) risk process with fixed transaction costs on dividends, \cite{Bayraktar14} proved that a $(u_{1},u_{2})$ strategy is again the OID. For more results on impulse dividend control problems, we refer readers to \cite{Hernandez18}-\cite{Yao11} and the references therein.

In the literature, capital injection is another factor to consider when designing dividend payout. Under risk models with dividends as well as fixed transaction costs imposed on the capital injections, the corresponding optimization problem is also an impulse (capital injection) control problem. In the setting of the dual classical CL risk model, \cite{Yao11} found that the \emph{optimal dividend and capital injection} (ODCI) strategy, which maximizes the expected present value of the dividends subtracted by the discounted cost of capital injections, pays out dividends according to a barrier strategy and injects capitals to bring the reserve up to a critical level whenever it falls below $0$.
Under the drifted diffusion risk model, \cite{Peng12} investigated the optimal dividend problem of an insurance company which controls risk exposure by reinsurance and by issuing new equity to protect the insurance company from bankruptcy. The corresponding ODCI strategy also pays dividends by a barrier strategy and injects capital to bring reserve up to a critical level whenever it falls below $0$. In the setting of SPL risk process with the dividend rate restricted,
\cite{Zhao17a} and \cite{Zhao17b} considered an ODCI problem and found that the optimal method of paying dividends is a threshold strategy.
For more information on dividend optimization in risk models with capital injection being imposed with proportional or fixed transaction cost, we refer readers to \cite{Zhao15}, \cite{Zhu17}-\cite{Bayraktar13}, and the references therein.

Regarding SNL risk processes, the majority of dividend optimization problems are formulated as non-impulse stochastic control problems. Using the expected present value of dividends until ruin (the expected present value of the dividends subtracted by the discounted costs of capital injections) as the value function, \cite{Avram07} identified the condition under which the barrier strategy (respectively, the barrier dividend strategy together with capital injection strategy that reflects the reserve process at $0$) is optimal among all admissible strategies. More results of non-impulse dividend optimization under the SNL risk processes can be found in \cite{Avram15}, \cite{Loeffen08}-\cite{Gerber69}, and the references therein. The non-impulse dividends optimization under the SPL risk processes can be found in \cite{Bayraktar14}, \cite{Avanzi11}, \cite{Zhao17a}, \cite{Zhao17b}, \cite{Avanzi17}-\cite{Avram07}, \cite{Kyprianou12}, \cite{Yin13}, and others.

Motivated by \cite{Loeffen09b} and \cite{Avram07}, this paper studies a general optimal IDCI problem through
maximizing the expected accumulated discounted net dividend payment subtracted by the accumulated discounted cost of injecting capital in the setup of the SNL risk process.
The novelty in this paper lies as follows: (i) compared with the existing OID results under diffusion or general L\'{e}vy setup, the present model  brings in the capital injection in an optimal way to reflect the corresponding risk process at $0$; and (ii) compared with the existing OID results concerning capital injections, the present model studies the L\'{e}vy setup, a more general driven process. In this paper, the discussion follows the standard treatment of \emph{Hamilton-Jacobi-Bellman} (HJB) inequality in the control theory. We first find the optimal strategy among all $(z_{1},z_{2})$ IDCI strategies, and then we prove that it is optimal among all IDCI strategies via a verification argument. To facilitate the standard HJB framework, we employ subtle approaches within each step, for example, the novel technique to derive Proposition \ref{Pro.2} and Lemma \ref{4.5}, and the mollifying argument to prove the modified verification lemma (see, Lemma \ref{Lem.1a} and Lemma \ref{Lem.1}).

We acknowledge that there is a parallel paper in the literature, \cite{Junca18}, which was also finished independently around the same time. The first version of both papers were available on internet in the middle of 2018. The authors of \cite{Junca18} considered the bail-out optimal dividend problem under fixed transaction costs for a L\'{e}vy risk model with a constraint on the expected net present value of injected capital. While the main results in this paper and those in \cite{Junca18} appear to be very similar, the primary objectives of these two papers are notably different as well as the methods adopted in the proof of certain main results (for instance, the verification Lemma \ref{Lem.1a} and Lemma \ref{Lem.1} in this paper vs Theorem 4.10 in \cite{Junca18}). We believe both papers make interesting contributions to the literature.

The remainder of this paper is organized as follows:
Section \ref{sec:2} comprises preliminaries concerning the SNL process and the mathematical setup of the dividend optimization problem. In Section \ref{sec:3} we represent the value function of a $(z_{1},z_{2})$ IDCI strategy using the scale function associated with the SNL process. This facilitates the characterization of the optimal strategy among all $(z_{1},z_{2})$ IDCI strategies, which is further proved to be optimal among all admissible IDCI strategies. In Section \ref{sec:4}, we first prove that a solution to the HJB inequalities coincides with the optimal value function via a verification lemma. Next, the solution to the HJB inequality is constructed, and the optimal strategy is found to be a $(z_{1},z_{2})$ IDCI strategy under which the risk process is reflected at $0$. In Section \ref{sec:5}, we also illustrate the optimal IDCI strategy by using one numerical example.

\section{Formulation of the dividend optimization problem}\label{sec:2}

Let  $X=\{X(t);t\geq0\}$ with probability  law $\{\mathrm{P}_{x};x\in\left[0,\infty\right)\}$ and natural filtration $\mathcal{F}=\{\mathcal{F}_{t};t\geq0\}$ be a SNL process, which is not a pure increasing linear drift or the negative of a sub-ordinator. Denote the running supremum  $\overline{X}(t):=\sup\{X(s);s\in [0,t]\}$ for $t\geq0$. Assume that in the case of no control (dividend is not deducted and capital is not injected), the risk process evolves as $X(t)$ for $t\geq0$.
An \emph{impulse dividend strategy}, denoted by $D=\{D(t);t\geq 0\}$, is a one-dimensional, non-decreasing,
left-continuous, $\mathcal{F}$-adapted, and pure jump process started at $0$, i.e., $D(0)=0$ and $D(t)$ defines the cumulative dividend that the company has paid out until time $t\ge 0$.
For the insurance company not to go bankrupt, the beneficiary of the dividend is required to inject capital into the insurance company to ensure that the  risk process is non-negative.
A \emph{capital injection strategy},  denoted by $R=\{R(t);t\geq 0\}$, is a one-dimensional, non-decreasing,
c\`{a}dl\`{a}g, $\mathcal{F}$-adapted process started at $0$, i.e., $R(0)=0$ and $R(t)$ defines
the cumulative capital that the beneficiary has injected until time $t\ge 0$. The combined pair $(D,R)$ is called an IDCI strategy. More explicitly, an impulse dividend strategy $D$ is characterized by
$$\left(\tau_n^{D},\eta_{n}^{D}\right),\quad n=1,2,\cdots,$$
where $\tau_{n}^{D}$ and $\eta_{n}^{D}$ are the $n$-th time and amount of dividend lump sum payment, respectively.
With dividends deducted according to $D$ and capital injected according to $R$, the controlled aggregate reserve process is then given by
\begin{eqnarray}\label{U}
U(t)=X(t)-D(t)+R(t), \quad t\geq 0.\nonumber
\end{eqnarray}
An IDCI strategy $(D,R)$ is defined to be \emph{admissible} if $U(t)\geq0$ for all $t\geq0$ and
$\int_{0}^{\infty}\mathrm{e}^{-qt}\mathrm{d}R(t)<\infty$ almost sure in the sense of $\mathrm{P}_{x}$,
where $q>0$ is a discount factor.

Let $\mathcal{D}$ be the set of all admissible
dividend and capital injection strategies.
For an IDCI strategy $(D,R)\in\mathcal{D}$, denote its value function as
$$
V_{(D,R)}(x)=\mathrm{E}_{x}\left(\sum_{n=1}^{\infty} \mathrm{e}^{-q \tau_{n}^{D}}\left(\eta_{n}^{D}-c\right)-\phi\int_{0}^{\infty}\mathrm{e}^{-qt}\mathrm{d}R(t)\right),\quad x\in[0,\infty),
$$
where  $c>0$ is the transaction cost for each lump sum dividend payment and $\phi>1$ is the cost of per unit capital injected. The goal is to identify the optimal strategy $(D^*,R^*)$ and the corresponding optimal value function
\begin{equation}\label{optimal function}
V(x)=V_{(D^*,R^*)}(x)=\sup\limits_{(D,R)\in\mathcal{D}}V_{(D,R)}(x),\quad x\in[0,\infty).\nonumber
\end{equation}
Intuitively speaking, because of $\phi>0$ and $q>0$, it would be better if the capital is injected as late as possible with no further capital injection being made rather than just injecting enough amounts to keep the corresponding risk process non-negative.

The Laplace exponent of $X$ is
\begin{eqnarray}
\psi(\theta)&=&\ln \mathrm{E}_{0}\left[\mathrm{e}^{\theta X(1)}\right]\nonumber
=\gamma\theta+\frac{1}{2}\sigma^{2}\theta^{2}-\int_{(0,\infty)}(1-\mathrm{e}^{-\theta x}-\theta x\mathbf{1}_{(0,1)}(x))\upsilon(\mathrm{d}x),\nonumber
\end{eqnarray}
where $\upsilon$ is  the L\'{e}vy measure with $\int_{(0,\infty)}(1\wedge x^{2})\upsilon(\mathrm{d}x)<\infty$.
Alternatively,
$$
X(t)=\gamma t+\sigma B(t)-\int_{0}^{t}\int_{(0,1)}x \overline{N}(\mathrm{d}s,\mathrm{d}x)-\int_{0}^{t}\int_{[1,\infty)}x N(\mathrm{d}s,\mathrm{d}x), \quad t\geq 0,
$$
where $B(t)$ is the standard Brownian motion, $N(\mathrm{d}s,\mathrm{d}x)$ is an independent Poisson random measure on $[0, \infty)\times (0, \infty)$ with intensity measure $\mathrm{d}s\upsilon(\mathrm{d}x)$, and $\overline{N}(\mathrm{d}s,\mathrm{d}x)=N(\mathrm{d}s,\mathrm{d}x)-\mathrm{d}s\upsilon(\mathrm{d}x)$ denotes the compensated random measure.

It is known that $\psi(\theta)<\infty$ for  $\theta\in[0,\infty)$, in which case it is strictly convex and infinitely differentiable.
As in \cite{Bertoin96}, the $q$-scale function of $X$,  for each $q\geq0$, $W^{(q)}:[0,\infty)\mapsto[0,\infty)$ is the unique strictly increasing and continuous function with Laplace transform
$$
\int_{0}^{\infty}\mathrm{e}^{-\theta x}W^{(q)}(x)\mathrm{d}x=\frac{1}{\psi(\theta)-q},\quad  \theta>\Phi_{q},
$$
where $\Phi_{q}$ is the largest solution of the equation $\psi(\theta)=q$. Further, let $W^{(q)}(x)=0 $ for $x<0$ and write $W$ for the $0$-scale function $W^{(0)}$.
For any $x\in\mathbb{R}$ and $\vartheta\geq0$, there exists the well-known exponential change of measure for an
SNL process
$$
\left.\frac{\mathrm{P}_{x}^{\vartheta}}{\mathrm{P}_{x}}\right|_{\mathcal{F}_{t}}=\mathrm{e}^{\vartheta\left(X(t)-x\right)-\psi(\vartheta)t}.
$$

Furthermore, under the probability measure $\mathrm{P}_{x}^{\vartheta}$,  $X$ remains an SNL
process and we denote by $W_{\vartheta}^{(q)}$ and $W_{\vartheta}$ respectively the $q$-scale function and the $0$-scale function for $X$ under $\mathrm{P}_{x}^{\vartheta}$.

Note that we do not impose the safety loading condition $\psi^{\prime}(0+)\geq 0$. Instead, $\psi^{\prime}(0+)>-\infty$ is assumed throughout the paper.

\section{The $(z_{1},z_{2})$ type dividend and capital injection strategy}\label{sec:3}

For the L\'{e}vy process $X$, denote the reflected process at infimum (or at $0$)
$$
Y(t)=X(t)-\inf_{0\leq s\leq t}\left(X(s)\wedge 0\right),\quad t\geq0.
$$
Define $T_{a}^{+}=\inf\{t\geq0;Y(t)> a\}$ and $\tau_{a}^{+}=\inf\{t\geq0;U(t)> a\}$, respectively, to be the up-crossing times of level $a\geq x$ of the processes $Y$ and $U$, with the convention $\inf\emptyset=\infty$. Define further
$$
\overline{W}^{(q)}(x)=\int_{0}^{x}W^{(q)}(z)\mathrm{d}z,\,\,
Z^{(q)}(x)=1+q \,\overline{W}^{(q)}(x),\,\,\overline{Z}^{(q)}(x)=\int_{0}^{x}Z^{(q)}(z)\mathrm{d}z.
$$
Then, for $x\in[0,b]$ and $q\geq0$, Proposition 2 of \cite{Pistorius04} holds that
\begin{equation}\label{two.side.exit.}
\mathrm{E}_x\left(\mathrm{e}^{-qT_{b}^{+}}\right)=Z^{(q)}(x)/Z^{(q)}(b).
\end{equation}

For $z_{1}<z_{2}$, let us consider an important type of IDCI strategy, such as the $(z_{1},z_{2})$ strategy $\{(D_{z_{1}}^{z_{2}}(t),R_{z_{1}}^{z_{2}}(t));t\geq0\}$:  a lump sum of dividend payment is made to bring the reserve level down to the level $z_{1}$ once the reserve hits or is above the level $z_{2}$, while no dividend payment is made whenever the reserve level is below $z_{2}$. Capital is injected in such a way that the reserve process is reflected at $0$
, i.e., $R_{z_{1}}^{z_{2}}(t)=-\inf\limits_{0\leq s\leq t}\left(X(s)-D_{z_{1}}^{z_{2}}(s)\right)\wedge 0$. To be precise, we define recursively
$T_{0}^{+}=0,\,\,T_{1}^{+}=T_{z_{2}}^{+}$ and
\begin{eqnarray}\label{Tn}
\hspace{-0.3cm}T^{+}_{n+1}&=&
\inf\bigg\{t> T^{+}_{n};
X(t)-(x\vee z_{2}-z_{1})-(n-1)(z_{2}-z_{1})
\nonumber\\
&&\hspace{-0.5cm}
-\inf\limits_{ s\leq t}\bigg[X(s)-\sum_{k=1}^{n-1}
\left(x\vee z_{2}-z_{1}+(k-1)(z_{2}-z_{1})\right)
\mathbf{1}_{(T_{k}^{+},T_{k+1}^{+}]}(s)
\nonumber\\
&&\hspace{-0.5cm}
-\left(x\vee z_{2}-z_{1}+(n-1)(z_{2}-z_{1})\right)
\mathbf{1}_{(T_{n}^{+},\infty)}(s)
\bigg]\wedge 0>z_{2}\bigg\},\,\, n\geq 1.
\end{eqnarray}
Then, the $(z_{1},z_{2})$ strategy can be re-expressed as
\begin{eqnarray}\label{D.repre.}
D_{z_{1}}^{z_{2}}(t)&=&\sum_{n=1}^{\infty}\left(x\vee z_{2}-z_{1}+(n-1)(z_{2}-z_{1})\right)
\mathbf{1}_{(T_{n}^{+},T_{n+1}^{+}]}(t)
,\quad t\geq 0,
\end{eqnarray}
and
\begin{eqnarray}
R_{z_{1}}^{z_{2}}(t)
&=&-\inf\limits_{ s\leq t}\bigg(X(s)-\sum_{n=1}^{\infty}
\left(x\vee z_{2}-z_{1}+(n-1)(z_{2}-z_{1})\right)
\nonumber\\
&&\hspace{1cm}\times
\mathbf{1}_{(T_{n}^{+},T_{n+1}^{+}]}(s)
\bigg)\wedge 0
,\quad t\geq 0.\nonumber
\end{eqnarray}

In the following result, the value function of a $(z_{1},z_{2})$ strategy, denoted by $V_{z_{1}}^{z_{2}}$, is expressed in terms of the scale functions.

\begin{prop}\label{Pro.1} Given $q>0$ and $c>0$, we have
\begin{eqnarray}\label{value.of.z1z2.1}
\hspace{-0.3cm}V_{z_{1}}^{z_{2}}(x)
&=&Z^{(q)}(x)\left(\frac{z_{2}\!-\!z_{1}\!-\!c}{Z^{(q)}(z_{2})\!-\!Z^{(q)}(z_{1})}
\!-\!\phi\frac{\overline{Z}^{(q)}(z_{2})\!-\!\overline{Z}^{(q)}(z_{1})}{Z^{(q)}(z_{2})\!-\!Z^{(q)}(z_{1})}\right)\nonumber\\
&&+\phi\left(\overline{Z}^{(q)}(x)+\frac{\psi^{\prime}(0+)}{q}\right),
\quad x\in[0,z_{2}),\,z_{1}+c\leq z_{2}<\infty,
\end{eqnarray}
and
\begin{eqnarray}\label{value.of.z1z2.2}
V_{z_{1}}^{z_{2}}(x)
&=&x+\frac{Z^{(q)}(z_{2})\left(z_{2}-z_{1}-c\right)}
{Z^{(q)}(z_{2})-Z^{(q)}(z_{1})}
-\phi\frac{\overline{Z}^{(q)}(z_{2})Z^{(q)}(z_{1})-\overline{Z}^{(q)}(z_{1})Z^{(q)}(z_{2})}{Z^{(q)}(z_{2})-Z^{(q)}(z_{1})}
\nonumber\\&&
-z_{2}+\phi\frac{\psi^{\prime}(0+)}{q},\quad x\in[z_{2},\infty),\,z_{1}+c\leq z_{2}<\infty.
\end{eqnarray}
\end{prop}

\begin{proof}\quad
Denote by $f(x)$ the expected discounted total lump sum dividend payments, and we have
$$
f(x)=
x-z_{1}-c+f(z_{1}),\quad x\in [z_{2},\infty),
$$
and
$$
f(x)=
\mathrm{E}_x\left(\mathrm{e}^{-q\tau_{z_{2}}^{+}}\right)
f(z_{2})
=\frac{Z^{(q)}(x)}{Z^{(q)}(z_{2})}(z_{2}-z_{1}-c+f(z_{1})),\quad x\in[0, z_{2}),
$$
which yield
$f(z_{1})=\frac{Z^{(q)}(z_{1})}{Z^{(q)}(z_{2})}(z_{2}-z_{1}-c+f(z_{1}))$, i.e.
$$
f(z_{1})=\frac{Z^{(q)}(z_{1})\left(z_{2}-z_{1}-c\right)}{Z^{(q)}(z_{2})-Z^{(q)}(z_{1})}.
$$
Hence
\begin{equation}\label{fx}
f(x)=\left\{\begin{array}{ll}
\frac{Z^{(q)}(x)\left(z_{2}-z_{1}-c\right)}{Z^{(q)}(z_{2})-Z^{(q)}(z_{1})},\quad & x\in[0,z_{2}),\\
x-z_{2}+\frac{Z^{(q)}(z_{2})\left(z_{2}-z_{1}-c\right)}{Z^{(q)}(z_{2})-Z^{(q)}(z_{1})},\quad & x\in[z_{2},\infty).
\end{array}\right.
\end{equation}

Denote by $g(x)$ the expected discounted total capital injections. By a similar argument as that which derives (4.8) of \cite{Avram07}, one gets for $x\in[0,z_{2}]$
\begin{eqnarray*}
g(x)&=&
\mathrm{E}_x\left(\int_{0}^{\tau_{z_{2}}^{+}}\mathrm{e}^{-qt}\mathrm{d}R(t)\right)
+\mathrm{E}_x\left(\mathrm{e}^{-q\tau_{z_{2}}^{+}}\right)g(z_{2})\\
&=&\mathrm{E}_x\left(\int_{0}^{\tau_{z_{2}}^{+}}\mathrm{e}^{-qt}\mathrm{d}R(t)\right)+\frac{Z^{(q)}(x)}{Z^{(q)}(z_{2})}g(z_{1})\\
&=&-\overline{Z}^{(q)}(x)-\frac{\psi^{\prime}(0+)}{q}+\left(\overline{Z}^{(q)}(z_{2})+\frac{\psi^{\prime}(0+)}{q}\right)
\frac{Z^{(q)}(x)}{Z^{(q)}(z_{2})}+\frac{Z^{(q)}(x)}{Z^{(q)}(z_{2})}g(z_{1}),
\end{eqnarray*}
which gives
\begin{equation}\label{gz1}
g(z_{1})=\frac{-\left(\overline{Z}^{(q)}(z_{1})+\frac{\psi^{\prime}(0+)}{q}\right)Z^{(q)}(z_{2})
+\left(\overline{Z}^{(q)}(z_{2})+\frac{\psi^{\prime}(0+)}{q}\right)Z^{(q)}(z_{1})}{Z^{(q)}(z_{2})-Z^{(q)}(z_{1})},\nonumber
\end{equation}
and thus
\begin{eqnarray}\label{gx< z_{2}}
g(x)
&=&\frac{Z^{(q)}(x)}{Z^{(q)}(z_{2})}\left(\overline{Z}^{(q)}(z_{2})\!-\!
\frac{\overline{Z}^{(q)}(z_{1})Z^{(q)}(z_{2})\!-\!\overline{Z}^{(q)}(z_{2})Z^{(q)}(z_{1})}{Z^{(q)}(z_{2})\!-\!Z^{(q)}(z_{1})}\right)\nonumber\\
&&\!-\overline{Z}^{(q)}(x)\!-\!\frac{\psi^{\prime}(0+)}{q},\quad x\in[0,z_{2}].
\end{eqnarray}
For $x\in(z_{2},\infty)$, by \eqref{gz1} we have
\begin{equation}\label{gx}
g(x)=g(z_{1})=\frac{-\overline{Z}^{(q)}(z_{1})Z^{(q)}(z_{2})+\overline{Z}^{(q)}(z_{2})Z^{(q)}(z_{1})}{Z^{(q)}(z_{2})-Z^{(q)}(z_{1})}-\frac{\psi^{\prime}(0+)}{q}.
\end{equation}
Collecting $V_{z_{1}}^{z_{2}}(x)=f(x)-\phi g(x)$, \eqref{fx}, \eqref{gx< z_{2}}, and \eqref{gx} yields \eqref{value.of.z1z2.1} and \eqref{value.of.z1z2.2} immediately. This completes the proof.
\qed
\end{proof}

Define, for $0<c\leq z_{1}+c< z_{2}<\infty$,
\begin{equation}\label{auxiliary.func.}
\xi(z_{1},z_{2})
=\frac{z_{2}-z_{1}-c}{Z^{(q)}(z_{2})-Z^{(q)}(z_{1})}-\phi\frac{\overline{Z}^{(q)}(z_{2})-\overline{Z}^{(q)}(z_{1})}{Z^{(q)}(z_{2})-Z^{(q)}(z_{1})}.
\end{equation}
Then, $V_{z_{1}}^{z_{2}}(x)=Z^{(q)}(x)\xi(z_{1},z_{2})+\phi\left(\overline{Z}^{(q)}(x)+\frac{\psi^{\prime}(0+)}{q}\right),\,\,
x\in[0,z_{2}]$.
The set of maximizers of $\xi(z_{1},z_{2})$ is written as
\begin{equation}\label{Def. set of maximizers}
\mathcal{M}:=\big\{(z_{1},z_{2});\,
c\leq z_{1}+c\leq z_{2},
\inf\limits_{x\geq0,\,
x+c\leq y}\left(\xi(z_{1},z_{2})-\xi(x,y)\right)\geq0\big\}.
\end{equation}

Denote by
\begin{equation}\label{f.u.p.t.r.}
\hat{\tau}_{z_{2}}=\inf\{t\geq0; \sup_{0\leq s\leq t}(X(s)\vee0)-X(t)>z_{2}\}
\end{equation}
the first passage time of the L\'{e}vy process reflected at its supremum.
The following result gives a useful link between the second derivative of $\xi$ (in $z_{2}$) and the Laplace transform of $\hat{\tau}_{z_{2}}$.

\begin{lem}
Let $\xi$ and $\hat{\tau}_{z_{2}}$ be defined respectively by \eqref{auxiliary.func.} and \eqref{f.u.p.t.r.}. We have
\begin{eqnarray}\label{25}
&&
\frac{\partial}{\partial z_2}\left(\frac{[Z^{(q)}(z_{2})-Z^{(q)}(z_{1})]^{2}}{qW^{(q)}(z_{2})}
\frac{\partial}{\partial z_{2}}\xi(z_{1},z_{2})\right)
\nonumber\\
&=&\frac{\left(Z^{(q)}(z_{2})-Z^{(q)}(z_{1})\right)W^{(q)\prime}(z_{2})}{[W^{(q)}(z_{2})]^{2}}
\left(-\frac{1}{q}+\frac{\phi}{q}\mathrm{E}_{0}\left(\mathrm{e}^{-q\hat{\tau}_{z_{2}}}\right)\right).
\end{eqnarray}
\end{lem}

\begin{proof}\quad
It follows from Proposition 2 (ii) of \cite{Pistorius04} that
\begin{eqnarray}\label{25.ax}
&&
\mathrm{E}_{0}\left(\mathrm{e}^{-q\hat{\tau}_{z_{2}}}\right)=
Z^{(q)}(z_{2})-
\frac{q[W^{(q)}(z_{2})]^{2}}{W^{(q)\prime}(z_{2})}
.
\end{eqnarray}
By algebraic manipulations one has
\begin{eqnarray*}
&&\frac{\partial}{\partial z_2}\left[\frac{\overline{Z}^{(q)}(z_{2})
-\overline{Z}^{(q)}(z_{1})}{Z^{(q)}(z_{2})-Z^{(q)}(z_{1})}\right]
\nonumber\\
&=&\frac{\frac{Z^{(q)}(z_{2})}{qW^{(q)}(z_{2})}
\left[Z^{(q)}(z_{2})-Z^{(q)}(z_{1})\right]-\overline{Z}^{(q)}(z_{2})+\overline{Z}^{(q)}(z_{1})}
{\left[Z^{(q)}(z_{2})-Z^{(q)}(z_{1})\right]^{2}\big/qW^{(q)}(z_{2})},
\end{eqnarray*}
and
\begin{eqnarray*}
& &\frac{\partial}{\partial z_2}\left(\frac{Z^{(q)}(z_{2})}{qW^{(q)}(z_{2})}
\left[Z^{(q)}(z_{2})-Z^{(q)}(z_{1})\right]-\overline{Z}^{(q)}(z_{2})+\overline{Z}^{(q)}(z_{1})\right)\\
&=&qW^{(q)}(z_{2})\frac{Z^{(q)}(z_{2})}{qW^{(q)}(z_{2})}
+[Z^{(q)}(z_{2})-Z^{(q)}(z_{1})]\frac{\partial}{\partial z_2}
\left[\frac{Z^{(q)}(z_{2})}{qW^{(q)}(z_{2})}\right]-Z^{(q)}(z_{2})\\
&=&\left(Z^{(q)}(z_{2})-Z^{(q)}(z_{1})\right)
\left(1-\frac{Z^{(q)}(z_{2})[W^{(q)}(z_{2})]{\color{red}^{\prime}}}{q[W^{(q)}(z_{2})]^{2}}\right)\\
&=&\left(Z^{(q)}(z_{2})-Z^{(q)}(z_{1})\right)\frac{W^{(q)\prime}(z_{2})}{[W^{(q)}(z_{2})]^{2}}
\left(\frac{[W^{(q)}(z_{2})]^{2}}{W^{(q)\prime}(z_{2})}-\frac{Z^{(q)}(z_{2})}{q}\right).
\end{eqnarray*}
One also has
$$
\frac{\partial}{\partial z_2}\left[\frac{z_{2}-z_{1}-c}{Z^{(q)}(z_{2})-Z^{(q)}(z_{1})}\right]
=\frac{\frac{Z^{(q)}(z_{2})-Z^{(q)}(z_{1})}{qW^{(q)}(z_{2})}-(z_{2}-z_{1}-c)}{\left[Z^{(q)}(z_{2})-Z^{(q)}(z_{1})\right]^{2}\big/qW^{(q)}(z_{2})},
$$
and
$$
\frac{\partial}{\partial z_2}\left[\frac{ Z^{(q)}(z_{2})-Z^{(q)}(z_{1})}{qW^{(q)}(z_{2})}-z_{2}+z_{1}+c\right]
=\frac{\left[Z^{(q)}(z_{2})-Z^{(q)}(z_{1})\right]W^{(q)\prime}(z_{2})}{-q[W^{(q)}(z_{2})]^{2}}.
$$
Combining the above facts, we obtain
\begin{eqnarray}\label{repr.of.deri.xi}
\frac{\partial}{\partial z_{2}}\xi(z_{1},z_{2})&=&\frac{1}{Z^{(q)}(z_{2})-Z^{(q)}(z_{1})}-\frac{[Z^{(q)}(z_{2})]^{\prime}(z_{2}-z_{1}-c)}{[Z^{(q)}(z_{2})-Z^{(q)}(z_{1})]^{2}}
\nonumber\\
&&
\hspace{-1cm}
-\phi \frac{Z^{(q)}(z_{2})}{Z^{(q)}(z_{2})-Z^{(q)}(z_{1})}
+\phi\frac{Z^{(q)\prime}(z_{2})[\overline{Z}^{(q)}(z_{2})-\overline{Z}^{(q)}(z_{1})]}{[Z^{(q)}(z_{2})-Z^{(q)}(z_{1})]^{2}},
\end{eqnarray}
which together with \eqref{25.ax} yields \eqref{25}.
\qed
\end{proof}

The following result characterizes
the optimal IDCI strategy among all $(z_{1},z_{2})$ strategies.

\begin{prop}\label{Pro.2}
Let $\xi$ be given by \eqref{auxiliary.func.}, then
$\mathcal{M}\neq\emptyset$ and, for $(z_{1},z_{2})\in\mathcal{M}$, we have
\begin{equation}\label{parc2.=0}
\frac{z_{2}-z_{1}-c}{Z^{(q)}(z_{2})-Z^{(q)}(z_{1})}-\phi\frac{\overline{Z}^{(q)}(z_{2})-\overline{Z}^{(q)}(z_{1})}{Z^{(q)}(z_{2})-Z^{(q)}(z_{1})}
=\frac{1-\phi Z^{(q)}(z_{2})}{qW^{(q)}(z_{2})}.
\end{equation}
\end{prop}

\begin{proof}\quad
By the definition of $\hat{\tau}_{z_{2}}$ one knows that it is increasing with respect to $z_{2}$ and $\lim\limits_{z_{2}\to\infty}\hat{\tau}_{z_{2}}=\infty$, implying $\lim\limits_{z_{2}\to\infty}\mathrm{E}_{0}\left(\mathrm{e}^{-q\hat{\tau}_{z_{2}}}\right)=0$. Hence there exists $\bar{z}_{0}\in[0,\infty)$ such that
\begin{equation}\label{upper.boun.of.brac.}
-\frac{1}{q}+\frac{\phi}{q}\mathrm{E}_{0}\left(\mathrm{e}^{-q\hat{\tau}_{z_{2}}}\right)\leq -\frac{1}{2q},
\quad z_{2}\in[\bar{z}_{0},\infty).
\end{equation}
On the other hand, we have
$$
\frac{W^{(q)}(z)}{W^{(q)\prime}(z)}=\frac{\mathrm{e}^{\Phi_{q}z}W_{\Phi_{q}}(z)}{[\mathrm{e}^{\Phi_{q}z}W_{\Phi_{q}}(z)]^{\prime}}
=\frac{1}{\Phi_{q}+\frac{W_{\Phi_{q}}^{\prime}(z)}{W_{\Phi_{q}}(z)}}\longrightarrow\frac{1}{\Phi_{q}},
$$
where $W^{(q)}(z)=\mathrm{e}^{\Phi_{q}z}W_{\Phi_{q}}(z)$ and $\lim\limits_{z\rightarrow\infty}\frac{W_{\Phi_{q}}^{\prime}(z)}
{W_{\Phi_{q}}(z)}=0$ (see \cite{Kyprianou06}, \cite{Loeffen09a} and Remark 4.5 of \cite{Wang18}).
Hence, by the rule of L'H\^{o}pital, we have
\begin{eqnarray*}
&&\lim_{z_{2}\to\infty}\frac{[Z^{(q)}(z_{2})-Z^{(q)}(z_{1})]W^{(q)\prime}(z_{2})}{[W^{(q)}(z_{2})]^{2}}
\nonumber\\
&=&\lim_{z_{2}\to\infty}\frac{Z^{(q)}(z_{2})-Z^{(q)}(z_{1})}{W^{(q)}(z_{2})}
\lim_{z_{2}\to\infty}\frac{W^{(q)\prime}(z_{2})}{W^{(q)}(z_{2})}\\
&=&\lim_{z_{2}\to\infty}\frac{qW^{(q)}(z_{2})}{W^{(q)\prime}(z_{2})}
\lim_{z_{2}\to\infty}\frac{W^{(q)\prime}(z_{2})}{W^{(q)}(z_{2})}=q.
\end{eqnarray*}
So, there exists $\bar{\bar{z}}_{0}\in[0,\infty)$ such that
$$
\frac{[Z^{(q)}(z_{2})-Z^{(q)}(z_{1})]W^{(q)\prime}(z_{2})}{[W^{(q)}(z_{2})]^{2}}\ge\frac{q }{2},\quad z_{2}\in[\bar{\bar{z}}_{0},\infty),
$$
which combined with \eqref{25} and \eqref{upper.boun.of.brac.} yields
$$
\frac{\partial}{\partial z_2}\left(\frac{[Z^{(q)}(z_{2})-Z^{(q)}(z_{1})]^{2}}{qW^{(q)}(z_{2})}
\frac{\partial}{\partial z_{2}}\xi(z_{1},z_{2})\right)
\leq -\frac{1}{4}, \quad z_{2}\in [\bar{z}_{0}\vee\bar{\bar{z}}_{0},\infty).
$$
Owing to  \eqref{repr.of.deri.xi}, it holds that
$$
\frac{[Z^{(q)}(\bar{z}_{0}\vee\bar{\bar{z}}_{0})-Z^{(q)}(z_{1})]^{2}}{qW^{(q)}(\bar{z}_{0}\vee\bar{\bar{z}}_{0})}
\frac{\partial}{\partial z_{2}}\xi(z_{1},\bar{z}_{0}\vee\bar{\bar{z}}_{0})<\infty.
  $$
Thus,  there exists $z_{0}\in(\bar{z}_{0}\vee\bar{\bar{z}}_{0},\infty)$ such that
$$
\frac{[Z^{(q)}(z_{2})-Z^{(q)}(z_{1})]^{2}}{qW^{(q)}(z_{2})}\frac{\partial}{\partial z_{2}}\xi(z_{1},z_{2})<0,\quad z_{2}\in [z_{0},\infty),
$$
which yields
$\frac{\partial}{\partial z_{2}}\xi(z_{1},z_{2})<0$ for $z_{2}\in [z_{0},\infty)$. As a result,
\begin{eqnarray}\label{boun.max.0}
\sup\limits_{0\leq z_{1},z_{2}<\infty,z_{1}+c\leq z_{2}}\xi(z_{1},z_{2})
=\sup\limits_{0\leq z_{1},z_{2}\leq z_{0},z_{1}+c\leq z_{2}}\xi(z_{1},z_{2}),
\end{eqnarray}
which, plus the continuity of $\xi$ over $\{(z_{1},z_{2});z_{1},z_{2}\in[0, z_{0}],z_{1}+c\leq z_{2}\}$,
yields
$$\emptyset\ne\mathcal{M}\subseteq \{(z_{1},z_{2});z_{1},z_{2}\in[0, z_{0}],z_{1}+c\leq z_{2}\}.$$

For IDCI strategies $(z_{1}, z_{2})$ and $(z_{1}^{\prime}, z_{2}^{\prime})$ with $z_{2}-z_{1}=z_{2}^{\prime}-z_{1}^{\prime}=c$ and $z_{2}^{\prime}>z_{2}$, let $T_{n}^{+\prime}$ be defined via \eqref{Tn} with $z_{i}$ replaced by $z_{i}^{\prime}$ (i=1,2). Then, by \eqref{Tn}, \eqref{D.repre.} and the technique of mathematical induction, we have,
for $x\in[0,z_{2}]$,
$$T_{n}^{+\prime}> T_{n}^{+},\quad n\geq1,$$
and hence  $D_{z_{1}^{\prime}}^{z_{2}^{\prime}}(t)\leq D_{z_{1}}^{z_{2}}(t)$ for $t\geq0$, which implies, for $t\geq0$,
$$
R_{z_{1}^{\prime}}^{z_{2}^{\prime}}(t)=
-\inf\limits_{0\leq s\leq t}[X(s)-D_{z_{1}^{\prime}}^{z_{2}^{\prime}}(s)]\wedge 0
\leq
-\inf\limits_{0\leq s\leq t}[X(s)-D_{z_{1}}^{z_{2}}(s)]\wedge 0
=R_{z_{1}}^{z_{2}}(t)
.$$
Therefore, we have
$$\mathrm{E}_x\left(\int_{0}^{\infty}\mathrm{e}^{-qs}\mathrm{d}R_{z_{1}}^{z_{2}}(s)\right)\geq \mathrm{E}_x\left(\int_{0}^{\infty}\mathrm{e}^{-qs}\mathrm{d}R_{z_{1}^{\prime}}^{z_{2}^{\prime}}(s)\right),\,\,
x\in[0,z_{2}]=[0,z_{2}]\cap[0,z_{2}^{\prime}],$$
which, combined with \eqref{gx< z_{2}} yields, for $z_{2}-z_{1}=z_{2}^{\prime}-z_{1}^{\prime}=c,\,z_{2}^{\prime}>z_{2}$,
\begin{equation}\label{rule.out.opti.on.the.line.}
\frac{\overline{Z}^{(q)}(z_{2})-\overline{Z}^{(q)}(z_{1})}{Z^{(q)}(z_{2})-Z^{(q)}(z_{1})}\geq \frac{\overline{Z}^{(q)}(z_{2}^{\prime})-\overline{Z}^{(q)}(z_{1}^{\prime})}{Z^{(q)}(z_{2}^{\prime})-Z^{(q)}(z_{1}^{\prime})}.
\end{equation}
By (\ref{rule.out.opti.on.the.line.}) and the definition of $\xi$ in (\ref{auxiliary.func.}), we may rule out the possibility that $\xi$ attains its  maximum value in the line $z_{2}=z_{1}+c$.
Indeed, if $(z_{1},z_{2})$ is a  maximum point of $\xi$ with $z_{2}=z_{1}+c$, then by (\ref{rule.out.opti.on.the.line.}) we should have $z_{2}=z_{1}=\infty$, contradicting \eqref{boun.max.0}.

Now, we have proved that $\emptyset\neq\mathcal{M}\subseteq \{(z_{1},z_{2});z_{1},z_{2}\in[0, z_{0}],z_{1}+c< z_{2}\}$.
Thus, if $(z_{1},z_{2})$ is a maximizer of $\xi(z_{1},z_{2})$, then it holds that $\frac{\partial}{\partial z_{2}}\xi(z_{1},z_{2})=0$, i.e., \eqref{parc2.=0} holds true.
\qed
\end{proof}

For an IDCI strategy $(z_{1},z_{2})\in\mathcal{M}$, the following result, an immediate consequence of  \eqref{value.of.z1z2.1}, \eqref{value.of.z1z2.2}, and \eqref{parc2.=0}, presents an alternative expression for the value function $V_{z_{1}}^{z_{2}}$.  It is interesting to see that this expression  is independent of $z_{1}$, which is not the case for arbitrary IDCI strategy $(z_{1},z_{2})$ (see \eqref{value.of.z1z2.1} and \eqref{value.of.z1z2.2}).

\begin{prop}\label{Pro.3} For $(z_{1},z_{2})\in\mathcal{M}$, the value function of the $(z_{1},z_{2})$ IDCI strategy
is
$$
\,\,V_{z_{1}}^{z_{2}}(x)=\left\{\begin{array}{ll}
\phi\left[\overline{Z}^{(q)}(x)+\frac{\psi^{\prime}(0+)}{q}\right]
+ Z^{(q)}(x)\frac{1-\phi Z^{(q)}(z_{2})}{qW^{(q)}(z_{2})},& x\in[0,z_{2}),\\
x-z_{2}+\phi\left[\overline{Z}^{(q)}(z_{2})+\frac{\psi^{\prime}(0+)}{q}\right]
+ Z^{(q)}(z_{2})\frac{1-\phi Z^{(q)}(z_{2})}{qW^{(q)}(z_{2})},& x\in[z_{2},\infty).
\end{array}\right.$$
\end{prop}

\begin{rem}
\label{rem1rem1}
Given $(z_{1},z_{2})\in\mathcal{M}$, one can verify that
$$
[V_{z_{1}}^{z_{2}}(x)]^{\prime}=\left\{\begin{array}{ll}
\phi Z^{(q)}(x)+W^{(q)}(x)\frac{1-\phi Z^{(q)}(z_{2})}{W^{(q)}(z_{2})}, &x\in[0,z_{2}),\\
1,&x\in(z_{2},\infty),
\end{array}\right.
$$
is continuous over $[0,\infty)$. If the scale function is differentiable, then one can also verify that
$$
[V_{z_{1}}^{z_{2}}(x)]^{\prime\prime}=\left\{\begin{array}{ll}
q\phi W^{(q)}(x)+[W^{(q)}(x)]^{\prime}\frac{1-\phi Z^{(q)}(z_{2})}{W^{(q)}(z_{2})}, &x\in[0,z_{2}),\\
0, &x\in(z_{2},\infty),
\end{array}\right.
$$
 is continuous on $[0,z_{2})$ and $(z_{2},\infty)$. However, $[V_{z_{1}}^{z_{2}}(x)]^{\prime\prime}$ is not evidently continuous at $z_{2}$. In fact, twice differentiability at $z_{2}$ is not guaranteed even if continuous differentiability is imposed on $W^{(q)}$.
Furthermore, if the scale function is only assumed to be piece-wise continuously differentiable over all compact subsets of $[0,\infty)$ (as in  Lemmas \ref{Lem.1a}-\ref{Lem.1} and Theorem \ref{HJB}), then $[V_{z_{1}}^{z_{2}}(x)]^{\prime\prime}$ is also piece-wise well-defined and piece-wise continuous.
\qed
\end{rem}

The following result characterizes several desirable properties of $V_{z_{1}}^{z_{2}}$ for $(z_{1},z_{2})\in\mathcal{M}$.

\begin{prop}\label{Pro.4}
Given $(z_{1},z_{2})\in\mathcal{M}$, $V_{z_{1}}^{z_{2}}$ is continuous and $[V_{z_{1}}^{z_{2}}]^{\prime}(x)\leq \phi$ over $[0,\infty)$, and
\begin{equation}\label{compar.property}
V_{z_{1}}^{z_{2}}(x)-V_{z_{1}}^{z_{2}}(y)\geq x-y-c,\quad 0\leq y< y+c\leq x.\nonumber
\end{equation}
\end{prop}

\begin{proof}\quad
We readily have $\phi Z^{(q)}(z_{2})+W^{(q)}(z_{2})\frac{1-\phi Z^{(q)}(z_{2})}{W^{(q)}(z_{2})}=1<\phi$.
By $W^{(q)}(0)\geq0$, $W^{(q)}(z_{2})>0$, and $1-\phi Z^{(q)}(z_{2})< 0$,
one can verify
$$
\phi Z^{(q)}(0)+W^{(q)}(0)\frac{1-\phi Z^{(q)}(z_{2})}{W^{(q)}(z_{2})}
=\phi +W^{(q)}(0)\frac{1-\phi Z^{(q)}(z_{2})}{W^{(q)}(z_{2})}\leq\phi.
$$
By Lemma 1 of \cite{Avram07} one has $W^{(q)}(x)\overline{W}^{(q)}(z_{2})\geq W^{(q)}(z_{2})\overline{W}^{(q)}(x)$ for $x\in[0,z_{2}]$, which, combined with $\phi>1$ and $W^{(q)}(x)>0$ for $x\in(0,z_{2})$, yields
\begin{eqnarray*}
&&\phi-\left(\phi Z^{(q)}(x)+W^{(q)}(x)\frac{1-\phi Z^{(q)}(z_{2})}{W^{(q)}(z_{2})}\right)\\
&=&-q\phi\int_{0}^{x}W^{(q)}(y)\mathrm{d}y-W^{(q)}(x)\frac{1-\phi-q\phi \int_{0}^{z_{2}}W^{(q)}(y)\mathrm{d}y}{W^{(q)}(z_{2})}\\
&=&\frac{1}{W^{(q)}(z_{2})}\left(q\phi\left(W^{(q)}(x)\overline{W}^{(q)}(z_{2})- W^{(q)}(z_{2})\overline{W}^{(q)}(x)\right)
+W^{(q)}(x)(\phi -1)\right)\\
&>&0,\quad x\in(0,z_{2}).
\end{eqnarray*}
In combination with these arguments, we reach $[V_{z_{1}}^{z_{2}}(x)]^{\prime}\leq \phi$ for $x\in[0,z_{2}]$.



By \eqref{Def. set of maximizers}, $(z_{1},z_{2})\in\mathcal{M}$, \eqref{auxiliary.func.}, and (\ref{parc2.=0}), we have
\begin{eqnarray}
\label{req.ineq.}
&&\frac{x-y-c}{Z^{(q)}(x)-Z^{(q)}(y)}-\phi\frac{\overline{Z}^{(q)}(x)-\overline{Z}^{(q)}(y)}{Z^{(q)}(x)-Z^{(q)}(y)}\nonumber\\
&\leq&
\frac{z_{2}-z_{1}-c}{Z^{(q)}(z_{2})-Z^{(q)}(z_{1})}-\phi\frac{\overline{Z}^{(q)}(z_{2})-\overline{Z}^{(q)}(z_{1})}{Z^{(q)}(z_{2})-Z^{(q)}(z_{1})}\nonumber\\
&=&\frac{1-\phi Z^{(q)}(z_{2})}{qW^{(q)}(z_{2})},\quad 0\leq y,y+c\leq x<\infty,
\end{eqnarray}
from which one can get
\begin{eqnarray*}
&&V_{z_{1}}^{z_{2}}(x)-V_{z_{1}}^{z_{2}}(y)\\
&=&\phi\left(\overline{Z}^{(q)}(x)-\overline{Z}^{(q)}(y)\right)
+ \left(Z^{(q)}(x)-Z^{(q)}(y)\right)\frac{1-\phi Z^{(q)}(z_{2})}{qW^{(q)}(z_{2})}\\
&=& \left(Z^{(q)}(x)-Z^{(q)}(y)\right)\left(\frac{z_{2}-z_{1}-c}{Z^{(q)}(z_{2})-Z^{(q)}(z_{1})}
-\phi\frac{\overline{Z}^{(q)}(z_{2})-\overline{Z}^{(q)}(z_{1})}{Z^{(q)}(z_{2})-Z^{(q)}(z_{1})}\right)\\
&&+\phi\left(\overline{Z}^{(q)}(x)-\overline{Z}^{(q)}(y)\right)\\
&\geq&
\left(Z^{(q)}(x)-Z^{(q)}(y)\right)\left(\frac{x-y-c}{Z^{(q)}(x)-Z^{(q)}(y)}
-\phi\frac{\overline{Z}^{(q)}(x)-\overline{Z}^{(q)}(y)}{Z^{(q)}(x)-Z^{(q)}(y)}\right)\\
&&+\phi\left(\overline{Z}^{(q)}(x)-\overline{Z}^{(q)}(y)\right)\\
&=&x-y-c, \quad0\leq y\leq x\leq z_{2},\,y+c\leq x.
\end{eqnarray*}

Using \eqref{req.ineq.} once again, one can get
\begin{eqnarray*}
V_{z_{1}}^{z_{2}}(x)-V_{z_{1}}^{z_{2}}(y)
&=&x-z_{2}+\phi\left(\overline{Z}^{(q)}(z_{2})+\frac{\psi^{\prime}(0+)}{q}\right)
+ Z^{(q)}(z_{2})\frac{1-\phi Z^{(q)}(z_{2})}{qW^{(q)}(z_{2})}\\
&&-\phi\left(\overline{Z}^{(q)}(x)+\frac{\psi^{\prime}(0+)}{q}\right)
-Z^{(q)}(x)\frac{1-\phi Z^{(q)}(z_{2})}{qW^{(q)}(z_{2})}\\
&\geq&x-z_{2}+z_{2}-y-c, \quad x\geq z_{2}\geq  y,\, y+c\leq x.
\end{eqnarray*}

For $x\geq y\geq z_{2}$ with $y+c\leq x$,
$V_{z_{1}}^{z_{2}}(x)-V_{z_{1}}^{z_{2}}(y)=x-y> x-y-c$.
The proof is completed.
\qed
\end{proof}

\section{Characterization of the optimal IDCI strategy}\label{sec:4}

This section is devoted to verifying that an IDCI strategy $(z_{1},z_{2})\in\mathcal{M}$ serves as the optimal IDCI strategy dominating all other admissible IDCI strategies.

We first present a result characterizing the optimal value function $V$, which helps with motivating the verification Lemma \ref{Lem.1}.

\begin{prop}\label{pro.v1}
The function $V(x)$ is continuous and $V^{\prime}(x)\leq\phi$ over $[0,\infty)$.
Also, $V(y)-V(x)\geq y-x-c$ for $y\geq x\geq0$.
\end{prop}

\begin{proof}\quad
By definition, any admissible IDCI strategy associated with the initial reserve $x\geq0$ also serves as an admissible IDCI strategy associated with the initial reserve $y\geq x$. Then it follows from \eqref{optimal function} that $V$ is non-decreasing.
For any $ \varepsilon>0$ and $y>x$, one can find an admissible IDCI strategy $(D_{y}^{ \varepsilon},R_{y}^{ \varepsilon})$ such that
   \begin{equation}\label{33.v1}
   V(y)\leq V_{(D_{y}^{ \varepsilon},R_{y}^{ \varepsilon})}(y)+ \varepsilon,
   \end{equation}
   where $R_{y}^{ \varepsilon}(t)\geq-\inf_{0\leq s\leq t}\left(X(s)-D_{y}^{ \varepsilon}(s)\right)\wedge 0$ because the latter is the minimum amount of capital injection needed to keep the reserve (applying dividend strategy $D_{y}^{ \varepsilon}$) being non-negative.

Define the admissible IDCI strategy
   \begin{eqnarray}
\hspace{-0.7cm}\overline{D}_{x}^{ \varepsilon}(t)&=&0,
   \quad\overline{R}_{x}^{ \varepsilon}(t)=-\inf_{0\leq s\leq t}X(s)\wedge 0,\quad t\in[0,T_{y}^{+}),
   \nonumber\\
   \overline{D}_{x}^{ \varepsilon}(t)&=&\left(D_{y}^{ \varepsilon}\circ\theta_{T_{y}^{+}}\right)(t-T_{y}^{+})
   ,\quad t\geq T_{y}^{+},
   \nonumber\\
   \overline{R}_{x}^{ \varepsilon}(t)
   &=&\overline{R}_{x}^{ \varepsilon}(T_{y}^{+}-)-\inf_{0\leq s\leq t-T_{y}^{+}}\left(\left(X-D_{y}^{ \varepsilon}\right)\circ\theta_{T_{y}^{+}}\right)(s)\wedge 0
   \nonumber\\
   &\leq&\overline{R}_{x}^{ \varepsilon}(T_{y}^{+}-)+\left(R_{y}^{ \varepsilon}\circ\theta_{T_{y}^{+}}\right)(t-T_{y}^{+}),\quad t\geq T_{y}^{+},
   \nonumber
   \end{eqnarray}
 where $\theta_{\cdot}$ is the time-shift operator. Then,
   $\left(\overline{D}_{x}^{ \varepsilon},\overline{R}_{x}^{ \varepsilon}\right)$ is indeed an admissible IDCI strategy associated with the initial reserve $x$. Denote
   \begin{eqnarray}
   &&C(\phi,\psi)=\phi\left(-\left(\overline{Z}^{(q)}(x)+\frac{\psi^{\prime}(0+)}{q}\right)+\left(\overline{Z}^{(q)}(y)+\frac{\psi^{\prime}(0+)}{q}\right)\frac{Z^{(q)}(x)}{Z^{(q)}(y)}
   \right),\nonumber\\
   &&\Delta\overline{D}_{x}^{ \varepsilon}(t)=\overline{D}_{x}^{ \varepsilon}(t+)-\overline{D}_{x}^{ \varepsilon}(t),\,\, \Delta D_{y}^{ \varepsilon}(t)=D_{y}^{ \varepsilon}(t+)-D_{y}^{ \varepsilon}(t),
   \nonumber\\
   &&\overline{G}_{x}^{ \varepsilon}(t)=\sum_{s\leq t}[\Delta\overline{D}_{x}^{ \varepsilon}(s)-c\mathbf{1}_{\{\Delta\overline{D}_{x}^{ \varepsilon}(s)>0\}}],\,\, G_{y}^{ \varepsilon}(t)=\sum_{s\leq t}[\Delta D_{y}^{ \varepsilon}(s)-c\mathbf{1}_{\{\Delta D_{y}^{ \varepsilon}(s)>0\}}].
   \nonumber
   \end{eqnarray}
By \eqref{two.side.exit.} as well as (4.4) in \cite{Avram07}, we have
$$\mathrm{E}_x\left(\int_{0}^{T_{y}^{+}}\mathrm{e}^{-qt}\mathrm{d}\overline{R}_{x}^{ \varepsilon}(t)\right)=\frac{C(\phi,\psi)}{\phi},$$
which, combined with \eqref{two.side.exit.} and \eqref{33.v1}, yields
   \begin{eqnarray}\label{34.v1}
   &&V(x)\ge V_{(\overline{D}_{x}^{ \varepsilon},\overline{R}_{x}^{ \varepsilon})}(x)   \nonumber\\
   &=&\mathrm{E}_x\left(\int_{T_{y}^{+}}^{\infty}\mathrm{e}^{-qt}\mathrm{d}\overline{G}_{x}^{ \varepsilon}(t)-\phi\int_{T_{y}^{+}}^{\infty}\mathrm{e}^{-qt}\mathrm{d}\overline{R}_{x}^{ \varepsilon}(t)\right)
   -\phi \mathrm{E}_x\left(\int_{0}^{T_{y}^{+}}\mathrm{e}^{-qt}\mathrm{d}\overline{R}_{x}^{ \varepsilon}(t)\right)\nonumber\\
   &\geq&\mathrm{E}_x\left[\mathrm{E}_x\left[\mathrm{e}^{-qT_{y}^{+}}\int_{T_{y}^{+}}^{\infty}\mathrm{e}^{-q(t-T_{y}^{+})}
   \mathrm{d}(G_{y}^{ \varepsilon}-\phi R_{y}^{ \varepsilon})\circ\theta_{T_{y}^{+}}(t-T_{y}^{+})\bigg|\mathcal{F}_{T_{y}^{+}}\right]\right]- C(\phi,\psi)
   \nonumber\\
   &=&\mathrm{E}_x\left(\mathrm{e}^{-qT_{y}^{+}}\right)V_{(D_{y}^{ \varepsilon},R_{y}^{ \varepsilon})}(y)  - C(\phi,\psi)=\frac{Z^{(q)}(x)} {Z^{(q)}(y)}V_{(D_{y}^{ \varepsilon},R_{y}^{ \varepsilon})}(y)- C(\phi,\psi)\nonumber\\
   &\ge&\frac{Z^{(q)}(x)} {Z^{(q)}(y)}[V(y)- \varepsilon] - C(\phi,\psi).
   \end{eqnarray}
   Owing to the non-decreasing property of $V(x)$, it follows from \eqref{34.v1} that
  $$
   0\leq V(y)-V(x)\le \left(1-\frac{Z^{(q)}(x)} {Z^{(q)}(y)}\right)V(y)+ \varepsilon + C(\phi,\psi).
 $$
  By setting $ \varepsilon\downarrow0$ and then $y\downarrow x$ ($x\uparrow y$, respectively) in the above we reach  continuity of $V(x)$.

For any $\varepsilon>0$ and $y\geq x\geq0$, denote $(D_{x}^{\varepsilon},R_{x}^{\varepsilon})$ an admissible IDCI strategy associated with the initial reserve $x$ such that $V_{(D_{x}^{\varepsilon},R_{x}^{\varepsilon})}(x)> V(x)-\varepsilon$.
Without loss of generality, $D_{x}^{\varepsilon}$ is expressed as
$$
\left(\tau_n^{D_{x}^{\varepsilon}},\eta_{n}^{D_{x}^{\varepsilon}}\right),\quad n=1,2,\cdots,$$
where $\tau_{n}^{D_{x}^{\varepsilon}}$ and $\eta_{n}^{D_{x}^{\varepsilon}}$ are respectively the $n$-th time and amount of dividend payout, and $\tau_{1}^{D_{x}^{\varepsilon}}>0$ a.s. .
Define a new admissible strategy $(D_{y}^{\varepsilon},R_{y}^{\varepsilon})$ associated with the initial reserve $y$ such that $R_{y}^{\varepsilon}=R_{x}^{\varepsilon}$ and $D_{y}^{\varepsilon}$ are characterized as
$$
\left(0,\tau_{1}^{D_{x}^{\varepsilon}},\tau_{2}^{D_{x}^{\varepsilon}},\cdots,\tau_{n}^{D_{x}^{\varepsilon}},\cdots;\,y- x,\eta_{1}^{D_{x}^{\varepsilon}},\eta_{2}^{D_{x}^{\varepsilon}},\cdots,\eta_{n}^{D_{x}^{\varepsilon}},\cdots\right).
$$
According to $(D_{y}^{\varepsilon},R_{y}^{\varepsilon})$, we have
$$
V(y)\ge V_{(D_{y}^{\varepsilon},R_{y}^{\varepsilon})}(y)
=y-x-c+V_{(D_{x}^{\varepsilon},R_{x}^{\varepsilon})}(x)>y-x-c+V(x)-\varepsilon,
$$
which yields $V(y)-V(x)\geq y-x-c$ after setting $\varepsilon\downarrow0$.

The inequality $V^{\prime}(x)\leq\phi$ over $[0,\infty)$ can be proved if we have
$$V(x)-V(y)\leq \phi (x-y),\quad x,y\in[0,\infty),$$
which, in the case of $x>y$ (proof for the case of $x<y$ is quite similar), can be accomplished by considering an IDCI strategy that injects a capital of amount $x-y$ at time $0$ to the reserve process starting from $y$.
\qed
\end{proof}

Put $\Delta D(t)=D(t+)-D(t)$, $\Delta X(t)=X(t)-X(t-)$, and $\Delta R(t)=R(t)-R(t-)$. Define
\begin{eqnarray}
\mathcal{D}_{1}&=&\{(D,R)\in\mathcal{D}; \,\Delta D(t)>0
\mathrm{\, \,\,iff\,\,\,}
c<\Delta D(t)\leq U(t-)+\Delta X(t),
\nonumber\\ &&
\,\,\,\,\mathrm{ and\,\,} R(t)=-\inf\limits_{0\leq s\leq t}\left(X(s)-D(s)\right)\wedge 0,\,\,t\geq0\},\nonumber
\end{eqnarray}
which is a proper subset of $\mathcal{D}$. Intuitively, the condition
$$c<\Delta D(t)\leq U(t-)+\Delta X(t),$$
requires that the lump sum dividend paid at time $t$ is strictly greater than $c$ and is no more than the available reserve after covering the down-ward jump of $X$ at time $t$, i.e., $U(t-)+\Delta X(t)$. For $(D,R)\in\mathcal{D}_{1}$, it is seen that $\Delta R(t)=0$ whenever $\Delta D(t)>0$.

The following result tells us that we can confine ourselves within $\mathcal{D}_{1}$ when searching for the optimal IDCI strategy among $\mathcal{D}$. This finding is used in the proof of the verification Lemma \ref{Lem.1}.

\begin{lem}\label{comf.}
The optimal IDCI strategy lies in $\mathcal{D}_{1}$.
\end{lem}

\begin{proof}\quad
For any $(D,R)\in\mathcal{D}\setminus \mathcal{D}_{1}$, one needs to find a $(\overline{D},\overline{R})\in\mathcal{D}_{1}$ such that
$V_{(D,R)}(x)< V_{(\overline{D},\overline{R})}(x),\,\,x\in[0,\infty)$.
For this purpose, let
\begin{eqnarray}
\overline{D}(t)&=&\sum_{s\in[0,t)}(\Delta D(s)\wedge (U(s-)+\Delta X(s)))
\mathbf{1}_{\{\Delta D(s)\wedge (U(s-)+\Delta X(s))>c\}},\quad t\geq0,\nonumber\\
\overline{R}(t)&=&-\inf\limits_{0\leq s\leq t}\left(X(s)-\overline{D}(s)\right)\wedge 0,\quad t\geq0,\nonumber
\end{eqnarray}
then one can see that $(\overline{D},\overline{R})\in\mathcal{D}_{1}$ and
\begin{eqnarray}\label{vit.eq.}
-\inf\limits_{0\leq s\leq t}\left(X(s)-D(s)\right)\wedge 0=\overline{R}(t)+D(t+)-\overline{D}(t+),\quad t\geq0.
\end{eqnarray}
Indeed, by the definition of $\overline{R}$, we have
$$X(t)-D(t)+\overline{R}(t)+D(t)-\overline{D}(t)
=X(t)+\overline{R}(t)-\overline{D}(t)\geq 0,\quad t\geq 0,$$
which implies
\begin{eqnarray}\label{EQALI}
\overline{R}(t)+D(t)-\overline{D}(t)\geq -\inf\limits_{0\leq s\leq t}\left(X(s)-D(s)\right)\wedge 0,\quad t\geq 0.
\end{eqnarray}
At the same time, one has
\begin{eqnarray}
&&X(t)-\overline{D}(t)+\overline{R}(t)-\left(\overline{R}(t)+D(t)-\overline{D}(t)+\inf\limits_{0\leq s\leq t_{0}}\left(X(s)-D(s)\right)\wedge 0\right)
\nonumber\\
&=&X(t)-D(t)-\inf\limits_{0\leq s\leq t}\left(X(s)-D(s)\right)\wedge 0
\geq 0, \quad t\geq0.
\nonumber
\end{eqnarray}
It implies that
\begin{eqnarray}
&&\overline{R}(t)-\left(\overline{R}(t)+D(t)-\overline{D}(t)+\inf\limits_{0\leq s\leq t}\left(X(s)-D(s)\right)\wedge 0\right)
\geq\overline{R}(t), \quad t\geq0,
\nonumber
\end{eqnarray}
which combined with \eqref{EQALI} and a choice of c\`{a}dl\`{a}g version gives \eqref{vit.eq.}.
By \eqref{vit.eq.} we have
\begin{eqnarray}
\label{ineq.ho.str.}
&&\overline{R}(t)
+
\sum_{s\in[0,t]}\left(\Delta D(s)-(U(s-)+\Delta X(s))\right)
\mathbf{1}_{\{\Delta D(s)>U(s-)+\Delta X(s)>c\}}
\nonumber\\
&&+\sum_{s\in[0,t]}\left(\Delta D(s)-\left(U(s-)+\Delta X(s)\right)\vee 0\right)
\mathbf{1}_{\{\Delta D(s)>c\geq U(s-)+\Delta X(s)\}}
\nonumber\\
&\leq&-\inf_{0\leq s\leq t}\left(X(s)-D(s)\right)\wedge 0\leq R(t),\quad t\geq0.
\end{eqnarray}
Furthermore, there is a $t_{0}\in(0,\infty)$ such that
$\overline{D}(t_{0})< D(t_{0})$ or
at least one inequality in \eqref{ineq.ho.str.} is a strict inequality at $t=t_{0}$. Otherwise, $(D,R)=(\overline{D},\overline{R})\in\mathcal{D}_{1}$, contradicting the fact that $(D,R)\notin\mathcal{D}_{1}$.
The above arguments combined with \eqref{vit.eq.} and the definition of $V_{(D,R)}$ yield
\begin{eqnarray}
&&
V_{(D,R)}(x)-V_{(\overline{D},\overline{R})}(x)
\nonumber\\
&\leq&
\mathrm{E}_{x}\left[\sum_{t\geq0} \mathrm{e}^{-qt}
\left(1-\phi\right)\left(\Delta D(t)-(U(t-)+\Delta X(t))\right)
\mathbf{1}_{\{\Delta D(t)> U(t-)+\Delta X(t)> c\}}
\right.
\nonumber\\
&&
+\sum_{t\geq0} \mathrm{e}^{-qt}
\left(1-\phi\right)\left(\Delta D(t)-\left(U(t-)+\Delta X(t)\right)\vee 0\right)
\mathbf{1}_{\{U(t-)+\Delta X(t)\leq c<\Delta D(t)\}}
\nonumber\\
&&
+\sum_{t\geq0} \mathrm{e}^{-qt}\left(\left(U(t-)+\Delta X(t)\right)\vee 0-c\right)
\mathbf{1}_{\{U(t-)+\Delta X(t)\leq c<\Delta D(t)\}}
\nonumber\\
&&
+\sum_{t\geq0} \mathrm{e}^{-qt}
\left(\Delta D(t)-c\right)
\mathbf{1}_{\{\Delta D(t)\leq c\}}
\nonumber\\
&&
-\phi \int_{0}^{\infty}\mathrm{e}^{-qt}\mathrm{d}\left(R(t)+
\inf_{0\leq s\leq t}\left(X(s)-D(s)\right)\wedge 0\right)\Bigg]
<0,
\nonumber
\end{eqnarray}
which completes the proof.
\qed
\end{proof}

As pointed out in Remark \ref{rem1rem1},
even if the continuous differentiability over $[0,\infty)$ is assumed on $W^{(q)}$, the twice differentiability of $V_{z_{1}}^{z_{2}}$ at $z_{2}$ is still absent in general, as is the continuity of $[V_{z_{1}}^{z_{2}}]^{\prime\prime}$ at $z_{2}$.
Furthermore, imposing on $W^{(q)}$ the assumption of continuous differentiability over $[0,\infty)$ will exclude important sub-classes of spectrally negative L\'{e}vy processes. For example, for a spectrally negative compound Poisson process which has jumps of exact size
$\alpha\in(0,\infty)$, the arrival rate $\lambda >0$, and a positive drift $\beta > 0$ such that $\beta-\lambda \alpha> 0$, the corresponding 0-scale function is identified by \cite{Asmussen00} and \cite{Hubalek11} as
$$W(x)=\frac{1}{\beta}\sum_{n=1}^{[x/\alpha]}\mathrm{e}^{-\lambda(\alpha n-x)/\beta}\frac{1}{n!}(\lambda/\beta)^{n}(\alpha n-x)^{n},$$
with $[x/\alpha]$ being the integer part of $x/\alpha$.
Note that the above example of scale function corresponds to a L\'evy process that has sample paths of bounded variation and whose L\'evy measure has atoms; otherwise, the scale function should be continuously
differentiable over $(0,\infty)$.

In the sequel, it is assumed that $W^{(q)}$ is piece-wise continuously differentiable over  all compact subsets of $[0,\infty)$, i.e., for any $a\in(0,\infty)$, $W^{(q)}$ is continuously differentiable over $[0,a]\setminus (d_{i}^{a})_{i\leq m_{a}}$, with $m_{a}$ being a non-negative integer. Hence, recalling that $V_{z_{1}}^{z_{2}}(x)$ is linear over $[z_{2},\infty)$, one knows that $V_{z_{1}}^{z_{2}}(x)$ is twice continuously differentiable over $[0,\infty)\setminus (d_{i})_{i\leq m}$ with $(d_{i})_{i\leq m}\subseteq [0,z_{2}]$ and $m$ being an integer.

For any function $f\in C^{2}((-\infty,\infty)\setminus (d_{i})_{i\leq m})$ for some integer $m\geq0$, define an operator $\mathcal{A}$ acting on $f$ as
$$
\mathcal{A}f(x)=\gamma f^{\prime}(x)+\frac{1}{2}\sigma^{2}f^{\prime\prime}(x)
+\int_{(0,\infty)}\left(f(x-y)-f(x)+f^{\prime}(x) y\mathbf{1}_{(0,1)}(y)\right)\upsilon(\mathrm{d}y),
$$
for $x\in(-\infty,\infty)\setminus (d_{i})_{i\leq m}$.
Also, define a sequence of mollified functions $f_{n}$ (of $f$) as
\begin{eqnarray}\label{Def.of.fn.0}
\hspace{-0.5cm}
f_{n}(x)&\hat{=}&
\int_{-\infty}^{+\infty}\rho_{n}(x-y)f(y)\mathrm{d}y\nonumber\\
&=&\int_{-\infty}^{+\infty}\rho(z)f(x-\tfrac{z}{n})
\mathrm{d}z,\quad x\in(-\infty,\infty),\,n\geq 1,
\end{eqnarray}
where $\rho_{n}(x)=n\rho(nx)$ and $\rho(x)=c\,\mathrm{e}^{\frac{1}{(x-1)^{2}-1}}\,\mathbf{1}_{(0,2)}(x)$ with  $\int_{-\infty}^{+\infty}\rho(x)\mathrm{d}x=1$.

In order to verify the optimality of a particular IDCI strategy $(z_{1},z_{2})\in\mathcal{M}$ producing the value function $V_{z_{1}}^{z_{2}}$, which lacks twice continuous differentiability at finitely many points $(d_{i})_{i\leq m}$, we need a modified version of verification argument, i.e., Lemma \ref{Lem.1}. Before presenting this verification argument, we show the following Lemma \ref{Lem.1a}.

\begin{lem}
\label{Lem.1a}
Let $f$ be a function such that
%
\begin{eqnarray}\label{unif.boun.s.d.f.0}
f\in C(-\infty,\infty)\cap C^{2}((-\infty,\infty)\setminus(d_{i})_{i\leq m}),
\end{eqnarray}
and
\begin{eqnarray}\label{unif.boun.s.d.f.1}
\max\limits_{i\leq m}\left(\lim\limits_{x\uparrow d_{i}}\big|f^{\prime\prime}(x)\big|\vee\lim\limits_{x\downarrow d_{i}}\big|f^{\prime\prime}(x)\big|\right)<\infty,
\end{eqnarray}
with $(d_{i})_{i\leq m}\subseteq (-\infty,\infty)$ and $m$ being a non-negative integer.
Suppose that
\begin{eqnarray}
f(x_{2})-f(x_{1})\geq x_{2}-x_{1}-c,\,\, f^{\prime}(x)\leq \phi,
\quad  x_2\ge x_1+c,\,x_{1},x\geq0,\label{Optimality conditionplus.0.0}
\end{eqnarray}
and
\begin{equation}
\mathcal{A}f(x)-q f(x)\leq 0,\quad x\in[0,\infty)\setminus(d_{i})_{i\leq m},\label{Optimality condition.0}
\end{equation}
then $(f_{n})_{n\geq1}$ defined through \eqref{Def.of.fn.0} are twice differentiable over $(-\infty,\infty)$,  satisfy \eqref{Optimality conditionplus.0.0} and
\begin{equation}
\mathcal{A}f_{n}(x)-q f_{n}(x)\leq 0,\quad x\in[0,\infty).\label{Optimality condition.0r}
\end{equation}
and
\begin{equation}\label{38.v4}
\lim\limits_{n\to\infty}f_{n}(x)=f(x),\quad x\in (-\infty,\infty).
\end{equation}
\end{lem}

\begin{proof}\quad
It is a direct consequence of differentiation under
the integral sign. The proof is omitted. \qed
\end{proof}

Now we are ready to present the following verification argument. For this purpose, let $(D^*, R^*)$ be a candidate optimal admissible IDCI strategy with value function $V_{(D^*,R^*)}(x),\,\, x\in[0,\infty)$. We extend the domain of $V_{(D^*,R^*)}$ to the entire real line by setting $V_{(D^*,R^*)}(x) = V_{(D^*,R^*)}(0)+\phi x$ for $x<0$. With a little abuse of notation, the extended function is still denoted by $V_{(D^*,R^*)}$.

\begin{lem}[Verification]\label{Lem.1}
If the function $V_{(D^*,R^*)}$ defined over $(-\infty,\infty)$
fulfills \eqref{unif.boun.s.d.f.0}, \eqref{unif.boun.s.d.f.1}, \eqref{Optimality conditionplus.0.0},  and \eqref{Optimality condition.0},
then $(D^*, R^*)$ is the optimal strategy, and $V_{(D^*,R^*)}(x)=\sup_{(D,R)\in \mathcal{D}} V_{(D,R)}(x),\,\,x\in[0,\infty)$.
 \end{lem}

\begin{proof}\quad By Lemma \ref{comf.}, we only need to prove that $(D^*, R^*)$ dominates all strategies among $\mathcal{D}_{1}$.
Let $\{U_{c}(t);t\geq0\}$ and $\{R_{c}(t);t\geq0\}$ be the continuous part of  $\{U(t);t\geq0\}$ and $\{R(t);t\geq0\}$, respectively. Let, for $N\geq 1,$
$$T_{N}=\inf\{t\geq0; U(t)>N\}$$
be the sequence of localization stopping times. Then, for all $t<T_{N}$, it holds that
\begin{equation}\label{U.bounded.up.down}
0\leq U(t)\leq N,
\end{equation}
i.e., both $U(t-)$ and $U(t)$ are restricted to the bounded compact set $\left[0,N\right]$.

Let $(f_{n})_{n\geq1}$ be defined via \eqref{Def.of.fn.0} with $f$ replaced by $V_{(D^*,R^*)}$. Hence, $f_{n}$ is twice differentiable over $[0,\infty)$, and satisfies \eqref{Optimality conditionplus.0.0} and \eqref{Optimality condition.0r}.
By Theorem 4.57 (It\^{o}'s formula) in \cite{Jacod02}, we have, for $x\in(0,\infty)$ and $n\geq1$,
\begin{eqnarray}\label{2.2}
&&\mathrm{e}^{-q (t\wedge T_{N})}f_{n}\left(U(t\wedge T_{N})\right)\nonumber\\
&=&f_{n}(x)-\int_{0-}^{t\wedge T_{N}}q \mathrm{e}^{-q s}f_{n}(U(s))\mathrm{d}s+\int_{0-}^{t\wedge T_{N}} \mathrm{e}^{-q s}f_{n}^{\prime}(U(s))\mathrm{d}U(s)\nonumber\\
&&+\frac{1}{2} \int_{0-}^{t\wedge T_{N}} \mathrm{e}^{-q s}f_{n}^{\prime\prime}(U(s))\mathrm{d}\langle U_{c}(\cdot),U_{c}(\cdot)\rangle_{s}
\nonumber\\
&&+\sum_{s\leq t\wedge T_{N}}\mathrm{e}^{-q s}\left(f_{n}(U(s-)+\Delta U(s))-f_{n}(U(s-))-f_{n}^{\prime}(U(s-))\Delta U(s)\right)
\nonumber\\
&&+\sum_{s\leq t\wedge T_{N}}\mathrm{e}^{-q s}\left(f_{n}(U(s+))-f_{n}\left(U(s)\right)+f_{n}^{\prime}(U(s))\left( D(s+)-D(s)\right)\right)\nonumber\\
&=&f_{n}(x)+\int_{0-}^{t\wedge T_{N}}\mathrm{e}^{-q s}(\mathcal{A}-q)f_{n}(U(s))\mathrm{d}s
+\int_{0-}^{t\wedge T_{N}}\sigma \mathrm{e}^{-q s}f_{n}^{\prime}(U(s))\mathrm{d}B(s)\nonumber\\
&&+\int_{0-}^{t\wedge T_{N}} \mathrm{e}^{-q s}f_{n}^{\prime}(U(s))\mathrm{d}R_{c}(s)\nonumber\\
&&+\int_{0-}^{t\wedge T_{N}}\int_{0}^{\infty}\mathrm{e}^{-q s}\left(f_{n}(U(s-)
-y)-f_{n}(U(s-))\right)
\overline{N}(\mathrm{d}s,\mathrm{d}y)\nonumber\\
&&+\int_{0-}^{t\wedge T_{N}}\int_{0}^{\infty}\mathrm{e}^{-q s}\left(f_{n}(U(s-)+
\Delta R(s)-y)-f_{n}(U(s-)-y)\right)N(\mathrm{d}s,\mathrm{d}y)\nonumber\\
&&+\sum_{s\leq t\wedge T_{N}} \mathrm{e}^{-q s}[f_{n}(U(s+))-f_{n}(U(s+)+D(s+)-D(s))],
\end{eqnarray}
where $\Delta U(s)=U(s)-U(s-)$ and $\Delta R(s)=R(s)-R(s-)$. Due to $(D,R)\in\mathcal{D}_{1}$, one knows that $\Delta R(s)>0$ implies a jump of $N(\cdot,\cdot)$ at time $s$ (i.e., whenever there is a jump in $R$, there must be a jump in $X$).
By \eqref{Optimality conditionplus.0.0} and the fact that  $D(s+)-D(s)>c$, we have, for $s\in\left[0,t\wedge T_{N}\right)$,
\begin{eqnarray}
&&\hspace{-0.6cm}f_{n}\left(U(s+)\right)-f_{n}\left(U(s+)+\left( D(s+)-D(s)\right)\right)+ D(s+)-D(s)-c\le 0,\label{2.3}\\
&&f_{n}(U(s-)+\Delta R(s)-y)-f_{n}(U(s-)-y)\le \phi\Delta R(s).\label{f'phi}
\end{eqnarray}
Therefore, by \eqref{Optimality conditionplus.0.0}, \eqref{Optimality condition.0r}, (\ref{2.2}), (\ref{2.3}), and (\ref{f'phi}), we have
\begin{eqnarray}\label{3.10}
&&\mathrm{e}^{-q (t\wedge T_{N})}f_{n}\left(U_{t\wedge T_{N}}\right)\nonumber\\
&\leq& f_{n}(x)+\phi\int_{0-}^{t\wedge T_{N}} \mathrm{e}^{-q s}\mathrm{d}R_{c}(s)\nonumber\\
&&+\phi\sum_{s\leq t\wedge T_{N}}\mathrm{e}^{-q s}\Delta R(s)+\int_{0-}^{t\wedge T_{N}}\sigma \mathrm{e}^{-q s}f_{n}^{\prime}(U(s))\mathrm{d}B(s)\nonumber\\
&&+\int_{0-}^{t\wedge T_{N}}\int_{0}^{\infty}\mathrm{e}^{-q s}\left(f_{n}(U(s-)-y)-f_{n}(U(s-))\right)
\overline{N}(\mathrm{d}s,\mathrm{d}y)\nonumber\\
&& -\sum\limits_{s\leq t\wedge T_{N}}\mathrm{e}^{-q s}( D(s+)-D(s)-c)
,\quad x\in(0,\infty),\quad n\geq1.
\end{eqnarray}
In addition, according to \cite{Ikeda81} (page 62), the compensated sum
$$t\mapsto \int_{0-}^{t\wedge T_{N}}\int_{0}^{\infty}\mathrm{e}^{-q s}\left(f_{n}(U(s-)
-y)-f_{n}(U(s-))\right)
\overline{N}(\mathrm{d}s,\mathrm{d}y)$$
is an $(\mathcal{F}_{t})$-martingale with zero mean. Indeed, the integrand of the above stochastic integration is bounded from below and above owing to (\ref{U.bounded.up.down}) and
\begin{eqnarray*}
\sup_{n\geq1}\sup_{x\in\left[0,N\right]}\left|f_{n}(x)\right|
\leq\sup_{w\in\left[0,N\right]}\left|f\left(w\right)\right|<\infty,
\end{eqnarray*}
where we have used $f(x)\in C[0,\infty)$ and the second equality in (\ref{Def.of.fn.0}).
Similarly, the integration
$$t\mapsto \int_{0-}^{t\wedge T_{N}}\sigma \mathrm{e}^{-q s}f_{n}^{\prime}(U(s))\mathrm{d}B(s)$$
is also an $(\mathcal{F}_{t})$-martingale with zero mean.

Taking expectations on both sides of \eqref{3.10}, we have
\begin{eqnarray}\label{45.v4}
f_{n}(x)&\geq&\mathrm{E}_x\left(\mathrm{e}^{-q (t\wedge T_{N})}f_{n}(U_{t\wedge T_{N}}^{D})\right)
-\phi \mathrm{E}_x\left(\int_{0-}^{t\wedge T_{N}} \mathrm{e}^{-q s}\mathrm{d}R(s)\right)\nonumber\\
&&+ \mathrm{E}_x\left(\sum_{s\leq t\wedge T_{N}}\mathrm{e}^{-q s}( D(s+)-D(s)-c)\right)
\nonumber\\
&\geq&\mathrm{E}_x\left(\mathrm{e}^{-q (t\wedge T_{N})}\left[\inf_{x\leq c}f_{n}(x)\right]\right)
-\phi \mathrm{E}_x\left(\int_{0-}^{t\wedge T_{N}} \mathrm{e}^{-q s}\mathrm{d}R(s)\right)\nonumber\\
&&+ \mathrm{E}_x\left(\sum_{s\leq t\wedge T_{N}}\mathrm{e}^{-q s}( D(s+)-D(s)-c)\right),\quad x\in(0,\infty),
\end{eqnarray}
where we have used the fact that, by \eqref{Optimality conditionplus.0.0}
$$f_{n}(U_{t\wedge T_{N}}^{D})
\geq [f_{n}(0)+U_{t\wedge T_{N}}^{D}-c]\mathbf{1}_{\{U_{t\wedge T_{N}}^{D}\geq c\}}
+\inf_{x\leq c}f_{n}(x)\mathbf{1}_{\{U_{t\wedge T_{N}}^{D}\leq c\}}
\geq\inf_{x\leq c}f_{n}(x).$$
By setting $n, t, N\to\infty$ in (\ref{45.v4}), and then using \eqref{38.v4} and the bounded convergence theorem ($\sup_{n\geq 1}\inf_{x\leq c}f_{n}(x)$ is bounded), we get
\begin{eqnarray}\label{46.v4}
f(x)&\geq&
-\phi \mathrm{E}_x\left(\int_{0-}^{\infty} \mathrm{e}^{-q s}\mathrm{d}R(s)\right)+ \mathrm{E}_x\left(\sum_{s}\mathrm{e}^{-q s}( D(s+)-D(s)-c)\right)\nonumber\\
&=&V_{(D,R)}(x),\quad x\in(0,\infty).\nonumber
\end{eqnarray}
By the arbitrariness of $(D,R)$, it follows that
$f(x)\geq \sup\limits_{(D,R)\in\mathcal{D}}V_{(D,R)}(x)$ for all $x\in(0,\infty)$,
which, along with the continuity of $f(x)$ and $\sup_{(D,R)\in\mathcal{D}}V_{(D,R)}(x)$ (see Proposition \ref{pro.v1}), implies
$$
f(x)\geq \sup_{(D,R)\in\mathcal{D}}V_{(D,R)}(x),\quad\forall\,x\in[0,\infty).
$$
Hence, by assumption we reach
$$
V_{(D^{*},R^{*})}(x)\geq \sup_{(D,R)\in\mathcal{D}}V_{(D,R)}(x),\quad \forall\,x\in[0,\infty).
$$
Since the reverse of the above inequality is trivial, we conclude with the desired equality. The proof is completed.
\qed
\end{proof}

\begin{rem}
Because $V_{(D^*,R^*)}$ lacks twice differentiability at $(d_{i})_{i\leq m}$, an appropriate generalized version of the It\^{o}'s lemma such as the
It\^{o}-Tanaka-Meyer formula (see \cite{Protter95}) should be applied to prove the verification argument.
In our case, we employed an alternative mollifying technique (see Lemmas \ref{Lem.1a} and \ref{Lem.1}) to deal with
the difficulty of lack of sufficient differentiability.

The mollifying arguments given in Lemmas \ref{Lem.1a} and \ref{Lem.1} are rigorous and differ from the approach adopted in \cite{Junca18} when proving their verification theorem.
\qed
\end{rem}

The following Lemma \ref{4.5} and Lemma \ref{4.6} are useful for characterizing the optimal IDCI strategy and the associated optimal value function in Theorem \ref{HJB}.

\begin{lem}\label{4.5}
Given $(z_{1},z_{2})\in\mathcal{M}$, we have
\begin{eqnarray}\label{decreasingsuff.cond.}
\hspace{-0.6cm}\phi+\frac{(1-\phi Z^{(q)}(x))[W^{(q)}(x)]^{\prime}}{q[W^{(q)}(x)]^{2}}
&\geq&
0,\quad x\in [z_{2},\infty).
\end{eqnarray}
\end{lem}

\begin{proof}\quad
It is seen that
\begin{eqnarray}\label{decreasingsuff.cond.00}
\hspace{-0.6cm}\phi+\frac{(1-\phi Z^{(q)}(x))[W^{(q)}(x)]^{\prime}}{q[W^{(q)}(x)]^{2}}&=&
\frac{\mathrm{d}}{\mathrm{d}x}\left(-\frac{1-\phi Z^{(q)}(x)}{qW^{(q)}(x)}\right)
\nonumber\\
&=&
-\frac{(\phi H(x)-1)[W^{(q)}(x)]^{\prime}}{q(W^{(q)}(x))^{2}}
,\, x\in[z_{2},\infty),
\end{eqnarray}
where $H(x)=Z^{(q)}(x)-q(W^{(q)}(x))^{2}/[W^{(q)}(x)]^{\prime}$.
By \eqref{25.ax} and $\lim\limits_{z_{2}\to\infty}\hat{\tau}_{z_{2}}=\infty$ we know that  $H(x)$ decreases in $x$ with $\lim_{x\to\infty}H(x)=0$.
Let $a_{0}>0$ be the unique zero of the function $\phi H(x)-1$ when $\phi H(0)>1$, then the inequality \eqref{decreasingsuff.cond.}
is equivalent to
\begin{equation}\label{equi.rela.}
z_{2}\geq \inf\{x>0; \frac{(\phi H(x)-1)[W^{(q)}(x)]^{\prime}}{q(W^{(q)}(x))^{2}} \leq0\}=\begin{cases}a_{0},\,\mbox{ when }\,\phi H(0)>1,\\
0,\,\,\,\,\mbox{ otherwise}.\nonumber
\end{cases}
\end{equation}
Since $z_{2}\geq 0$ holds trivially, we only need to show that $z_{2}\geq a_{0}$ holds when $\phi H(0)>1$.
Given $\phi H(0)>1$, by \eqref{decreasingsuff.cond.00} and the decreasing property of $H(x)$, the function $\frac{1-\phi Z^{(q)}(x)}{qW^{(q)}(x)}$ is increasing (decreasing) over $[0,a_{0})$ ($(a_{0},\infty)$), and attains its  maximum at $a_{0}$. So, when $\frac{1-\phi Z^{(q)}(z_{2})}{qW^{(q)}(z_{2})}=\frac{1-\phi Z^{(q)}(a_{0})}{qW^{(q)}(a_{0})}$ we must have $z_{2}= a_{0}$. Further, when $\frac{1-\phi Z^{(q)}(z_{2})}{qW^{(q)}(z_{2})}<\frac{1-\phi Z^{(q)}(a_{0})}{qW^{(q)}(a_{0})}$ we should have
$z_{2}>a_{0}$. Otherwise, $z_{2}$ will be in the range $(z_{1},a_{0})$, which leads to
\begin{eqnarray}
\frac{\partial}{\partial z_{1}}\xi(z_{1},z_{2})
&=&
\frac{qW^{(q)}(z_{1})}{Z^{(q)}(z_{2})-Z^{(q)}(z_{1})}
\left(\xi(z_{1},z_{2})-\frac{1-\phi Z^{(q)}(z_{1})}{qW^{(q)}(z_{1})}\right)
\nonumber\\
&=&
\frac{qW^{(q)}(z_{1})}{Z^{(q)}(z_{2})-Z^{(q)}(z_{1})}
\left(\frac{1-\phi Z^{(q)}(z_{2})}{qW^{(q)}(z_{2})}-\frac{1-\phi Z^{(q)}(z_{1})}{qW^{(q)}(z_{1})}\right)
\nonumber\\
&>&0.\nonumber
\end{eqnarray}
This result contradicts the fact that $\xi$ attains its  maximum at $(z_{1},z_{2})$, so $z_{2}\notin(z_{1},a_{0})$. The proof is completed.
\qed
\end{proof}

In the sequel, we extend the function $V_{z_{1}}^{z_{2}}$ to the entire real axis by setting $V_{z_{1}}^{z_{2}}(x) = V_{z_{1}}^{z_{2}}(0)+\phi x$ for $x<0$. We denote by $V_{x}(y)$ the value function of the barrier dividend and capital injection strategy with barrier level $x$ and initial reserve $y$ (cf., Equation (5.4) in \cite{Avram07}), i.e.,
\begin{equation}\label{barrierstra.Vxy}
\hspace{-0.05cm}
V_{x}(y)=\left\{\begin{array}{ll}
V_{x}(0)+\phi y,&y<0,\\
\phi(\overline{Z}^{(q)}(y)+\frac{\psi^{\prime}(0+)}{q})+ Z^{(q)}(y)\frac{1-\phi Z^{(q)}(x)}{qW^{(q)}(x)},&y\in[0,x),\\
y-x+\phi(\overline{Z}^{(q)}(x)+\frac{\psi^{\prime}(0+)}{q})+ Z^{(q)}(x)\frac{1-\phi Z^{(q)}(x)}{qW^{(q)}(x)}, & y\geq x.
\end{array}\right.
\end{equation}

\begin{lem}\label{4.6}
Given $(z_{1},z_{2})\in\mathcal{M}$ and $x\in(z_{2},\infty)$, define
$$h(z):=V_{z_{1}}^{z_{2}}(z)-V_{x}(z),\quad z\in(-\infty,x].$$
Then, $h(z)$ is non-decreasing with respect to $z$ and $h(x)\geq0$.
\end{lem}

\begin{proof}
By the Mean Value Theorem, for $x>z_{2}$ we have
\begin{eqnarray}\label{65}
h(x)&=&V_{z_{1}}^{z_{2}}(x)-V_{x}(x)
\nonumber\\
&=&\sum_{i=0}^{m}\left(-y+\phi\overline{Z}^{(q)}(y)
+ Z^{(q)}(y)\frac{1-\phi Z^{(q)}(y)}{qW^{(q)}(y)}\right)\bigg|_{x_{i+1}}^{x_{i}}\nonumber\\
&=&\sum_{i=0}^{m}Z^{(q)}(\theta_{i})\left(-\phi-\frac{(1-\phi Z^{(q)}(\theta_{i}))[W^{(q)}(\theta_{i})]^{\prime}}{q[W^{(q)}(\theta_{i})]^{2}}\right)(x_{i}-x_{i+1})\geq0,\nonumber
\end{eqnarray}
where $x_{0}=z_{2}$, $x_{m+1}=x$, $x_{i}=(z_{2}\vee d_{i})\wedge x$ for $i\in\{1,\cdots,m\}$, $\theta_{i}\in(x_{i},x_{i+1})$ and (\ref{decreasingsuff.cond.}) holds for $\theta_{i}\in(x_{i},x_{i+1})\subseteq(z_{2},x)$ whenever $x_{i}<x_{i+1}$.

By (\ref{decreasingsuff.cond.}) and the Mean Value Theorem, we have
\begin{eqnarray*}
h^{\prime}(z)&=&Z^{(q)}(z)\left(\frac{1-\phi Z^{(q)}(z_{2})}{qW^{(q)}(z_{2})}-\frac{1-\phi Z^{(q)}(x)}{qW^{(q)}(x)}\right)\\
&=&
Z^{(q)}(z)\sum_{i=0}^{m}\left(\frac{1-\phi Z^{(q)}(x_{i})}{qW^{(q)}(x_{i})}-\frac{1-\phi Z^{(q)}(x_{i+1})}{qW^{(q)}(x_{i+1})}\right)\\
&=&
Z^{(q)}(z)\sum_{i=0}^{m}\left(-\phi-\frac{(1-\phi Z^{(q)}(\theta_{i}))[W^{(q)}(\theta_{i})]^{\prime}}{q[W^{(q)}(\theta_{i})]^{2}}\right)(x_{i}-x_{i+1})
\ge0,
\end{eqnarray*}
for $z\in[0,z_{2})$;
\begin{eqnarray*}
h^{\prime}(z)&=&1-\phi Z^{(q)}(z)-W^{(q)}(z)\frac{1-\phi Z^{(q)}(x)}{W^{(q)}(x)}\nonumber\\
&=&qW^{(q)}(z)\left(\frac{1-\phi Z^{(q)}(z)}{qW^{(q)}(z)}-\frac{1-\phi Z^{(q)}(x)}{qW^{(q)}(x)}\right)\\
&=&qW^{(q)}(z)\sum_{i=0}^{m}\left(\frac{1-\phi Z^{(q)}(y_{i})}{qW^{(q)}(y_{i})}-\frac{1-\phi Z^{(q)}(y_{i+1})}{qW^{(q)}(y_{i+1})}\right)\\
&=&qW^{(q)}(z)\sum_{i=0}^{m}\left(-\phi-\frac{(1-\phi Z^{(q)}(\eta_{i}))[W^{(q)}(\eta_{i})]^{\prime}}{q[W^{(q)}(\eta_{i})]^{2}}\right)(y_{i}-y_{i+1})\ge0,
\end{eqnarray*}
for $z\in[z_{2},x)$, $y_{0}=z$, $y_{m+1}=x$, $y_{i}=(z\vee d_{i})\wedge x$, and $\eta_{i}\in(y_{i},y_{i+1})\subseteq(z,x)$, $i=1,\cdots,m$; and
$$
h^{\prime}(z)=\phi-\phi=0,\quad\mbox{for $z\in(-\infty,0)$.}
$$
The proof is completed.
\qed
\end{proof}

The following theorem characterizes the optimal IDCI strategy among all admissible IDCI strategies. The ideas in the proof are partly obtained from \cite{Avram07} and \cite{Loeffen08}. This theorem shows that any IDCI strategy $(z_{1},z_{2})\in\mathcal{M}$ is optimal and dominates all admissible IDCI strategies.

\begin{thm}\label{HJB}
Suppose that the scale function $W^{(q)}$ is piece-wise continuously differentiable over  all compact subsets of $[0,\infty)$.
Let $(z_{1},z_{2})\in\mathcal{M}$. Then the $(z_{1},z_{2})$ strategy is optimal among all admissible IDCI strategies.
\end{thm}

\begin{proof}\quad
By the fact that the scale function $W^{(q)}(x)$ is left and right differentiable over $(0,\infty)$ (see for example, Lemma 1 in \cite{Pistorius04}), Remark \ref{rem1rem1}, and the extended definition $V_{z_{1}}^{z_{2}}(x) = V_{z_{1}}^{z_{2}}(0)+\phi x$ for $x<0$, one can verify that $V_{z_{1}}^{z_{2}}$ satisfies \eqref{unif.boun.s.d.f.1}.

Let $\tau^{+}_{a}$ ($\tau^{-}_{b}$) be the first up-crossing (down-crossing) time of level $a$ ($b$) by the process $U $
\begin{equation}\label{taub-}
\tau_{a}^{+}:=\inf\{t>0;U(t)> a\},\quad\tau_{b}^{-}:=\inf\{t>0;U(t)\leq b\}.
\end{equation}

By Proposition \ref{Pro.4}, we need only to prove
$\mathcal{A}V_{z_{1}}^{z_{2}}(x)-q V_{z_{1}}^{z_{2}}(x)\leq0$ for $x\in[0,\infty)\setminus(d_{i})_{i\leq m}$. Here, $(d_{i})_{i\leq m}$ with $d_{0}:=0< d_{1}<\cdots<d_{m}<\infty:=d_{m+1}$ is the set of points where the continuous differentiability is absent for $W^{(q)}$.

Given $x\in(0,z_{2})\setminus(d_{i})_{i\leq m}$, without loss of generality we may assume $x\in(d_{i},d_{i+1})\cap (0,z_{2})$ for some $0\leq i\leq m$.
Let $\tau:=\tau_{d_{i}}^{-}\wedge \tau_{d_{i+1}\wedge z_{2}}^{+}$ with $\tau_{d_{i}}^{-}$ and $\tau_{d_{i+1}\wedge z_{2}}^{+}$ defined via (\ref{taub-}). By the strong Markov property of the process $X$, we have
\begin{eqnarray}
&&
\mathrm{E}_{x}\left(\int_{0}^{\infty}\mathrm{e}^{-qt}\mathrm{d}(D_{z_{1}}^{z_{2}}(t)-\phi R_{z_{1}}^{z_{2}}(t))\bigg|\mathcal{F}_{r\wedge \tau}
\right)\nonumber\\
&=&\mathrm{E}_{x}\left(\int_{0}^{\infty}\mathrm{e}^{-q(s+r\wedge \tau)}
\mathrm{d}(D_{z_{1}}^{z_{2}}(s+r\wedge \tau)-\phi R_{z_{1}}^{z_{2}}(s+r\wedge \tau))\bigg|\mathcal{F}_{r\wedge \tau}\right)
\nonumber\\
&=&\mathrm{e}^{-q(r\wedge \tau)}\mathrm{E}_{X(r\wedge \tau)}\left(\int_{0}^{\infty}\mathrm{e}^{-qs}
\mathrm{d}(D_{z_{1}}^{z_{2}}(s)-\phi R_{z_{1}}^{z_{2}}(s))\right)\nonumber\\
&=&\mathrm{e}^{-q(r\wedge \tau)}V_{z_{1}}^{z_{2}}(U(r\wedge \tau)), r\geq0,\nonumber
\end{eqnarray}
which implies that the right-hand side of the above equation is a martingale.
Here, we have used the fact that no dividends are paid out and no capital is injected during the time interval $[0,\tau]$; i.e., $U(r)= X(r)$ for $r\in[0,\tau]$.

The martingale property of the process $\left(\mathrm{e}^{-q\left(r\wedge \tau\right)}V_{z_{1}}^{z_{2}}\left(U\left(r\wedge \tau\right)\right)\right)_{r\geq0}$ implies
\begin{eqnarray}\label{ver.xles.a}
\mathcal{A}V_{z_{1}}^{z_{2}}(x)-q V_{z_{1}}^{z_{2}}(x)=0, \quad x\in(0,z_{2})\setminus (d_{i})_{i\leq m}.
\end{eqnarray}
Indeed, for $x\in(d_{i},d_{i+1})\cap (0,z_{2})$ and $\tau=\tau_{d_{i}}^{-}\wedge \tau_{d_{i+1}\wedge z_{2}}^{+}$, It\^{o}'s formula gives
\begin{eqnarray}
&&\mathrm{e}^{-q(r\wedge \tau)}V_{z_{1}}^{z_{2}}(U(r\wedge \tau))-V_{z_{1}}^{z_{2}}(x)\nonumber\\
&=&\int_{0-}^{r\wedge \tau}\mathrm{e}^{-q s}(\mathcal{A}-q)V_{z_{1}}^{z_{2}}(U(s))\mathrm{d}s+\int_{0-}^{r\wedge \tau}\sigma \mathrm{e}^{-q s}[V_{z_{1}}^{z_{2}}]^{\prime}(U(s))\mathrm{d}B(s)\nonumber\\
&&+\int_{0-}^{r\wedge \tau}\int_{0}^{\infty}\mathrm{e}^{-q s}[V_{z_{1}}^{z_{2}}(U(s-)-y)-V_{z_{1}}^{z_{2}}(U(s-))]
\overline{N}(\mathrm{d}s,\mathrm{d}y),\quad r\geq0.\nonumber
\end{eqnarray}
Taking expectations on both sides of the above display after localization, we have
$$
0=\mathrm{E}_{x}\left(\int_{0-}^{r\wedge \tau}\mathrm{e}^{-q s}(\mathcal{A}-q)V_{z_{1}}^{z_{2}}(U(s))\mathrm{d}s\right),\quad r\geq0.
$$
Being divided  by $r$ on both sides and then setting $r\downarrow0$ in the above equation, we get (\ref{ver.xles.a}) for $x\in(0,z_{2})\setminus (d_{i})_{i\leq m}$ by the mean value theorem together with the dominated convergence theorem. For a more detailed proof of (\ref{ver.xles.a}), we can also turn to Proposition 2.1 of \cite{Hojgaard99}. 
Thus, it suffices to further prove
\begin{eqnarray}\label{ver.xbig.a*}
\mathcal{A}V_{z_{1}}^{z_{2}}(x)-q V_{z_{1}}^{z_{2}}(x)\leq0, \quad x\in[z_{2},\infty)\setminus (d_{i})_{i\leq m}.
\end{eqnarray}

By using similar arguments as those used in proving (\ref{ver.xles.a}) we can get
$$\mathcal{A}V_{x}(y)-q V_{x}(y)=0, \quad y\in(0,x)\setminus (d_{i})_{i\leq m}, x\in(0,\infty),$$
which implies
\begin{eqnarray}\label{3.25}
\lim\limits_{y\uparrow x}\left(\mathcal{A}V_{x}(y)-q V_{x}(y)\right)=0,\quad x\in(z_{2},\infty)\setminus (d_{i})_{i\leq m},
\end{eqnarray}
where $\lim\limits_{y\uparrow x}\mathcal{A}V_{x}(y)$ is well-defined due to \eqref{barrierstra.Vxy}.
Meanwhile, because the function $\mathcal{A}V_{z_{1}}^{z_{2}}-q V_{z_{1}}^{z_{2}}$ is continuous over $(z_{2},\infty)$ (actually, $[V_{z_{1}}^{z_{2}}]^{\prime\prime}(x)=0$ for $x\in(z_{2},\infty)$, see Proposition \ref{Pro.3}), we have
\begin{eqnarray}\label{3.26}
\hspace{-0.5cm}
\lim\limits_{y\uparrow x}\left(\mathcal{A}V_{z_{1}}^{z_{2}}(y)-q V_{z_{1}}^{z_{2}}(y)\right)=\mathcal{A}V_{z_{1}}^{z_{2}}(x)-q V_{z_{1}}^{z_{2}}(x), x\in(z_{2},\infty)\setminus (d_{i})_{i\leq m}.
\end{eqnarray}
Combining (\ref{3.25}) and (\ref{3.26}), to prove (\ref{ver.xbig.a*}) it suffices to show
\begin{eqnarray}
\lim\limits_{y\uparrow x}\left(\mathcal{A}[V_{z_{1}}^{z_{2}}(y)-V_{x}(y)]-q[V_{z_{1}}^{z_{2}}(y)-V_{x}(y)]\right)\leq0,\quad x\in(z_{2},\infty)\setminus (d_{i})_{i\leq m}.\nonumber
\end{eqnarray}
For $x\in(z_{2},\infty)\setminus (d_{i})_{i\leq m}$, we can use the dominated
convergence theorem to deduce
\begin{eqnarray}\label{sign.lim.oper.}
&&\lim_{y\uparrow x}\left(\mathcal{A}[V_{z_{1}}^{z_{2}}(y)-V_{x}(y)]-q [V_{z_{1}}^{z_{2}}(y)-V_{x}(y)]\right)\nonumber\\
&=&\gamma\left([[V_{z_{1}}^{z_{2}}]^{\prime}(x)-V_{x}^{\prime}(x)]\right)+\frac{\sigma^{2}}{2}[[V_{z_{1}}^{z_{2}}]^{\prime\prime}(x)-\lim_{y\uparrow x}V_{x}^{\prime\prime}(y)]-q[V_{z_{1}}^{z_{2}}(x)-V_{x}(x)]\nonumber\\
&&+\int_{(0,\infty)}\left([V_{z_{1}}^{z_{2}}(x\!-\!y)\!-\!V_{x}(x\!-\!y)]\!-\![V_{z_{1}}^{z_{2}}(x)\!-\!V_{x}(x)]
\right.
\nonumber\\
&&
\left.
+[[V_{z_{1}}^{z_{2}}]^{\prime}(x)\!-\!V_{x}^{\prime}(x)] y\mathbf{1}_{(0,1)}(y)\right)\upsilon(\mathrm{d}y)\nonumber\\
&=&-\frac{\sigma^{2}}{2}\lim\limits_{y\uparrow x}V_{x}^{\prime\prime}(y)-q[V_{z_{1}}^{z_{2}}(x)-V_{x}(x)]\nonumber\\
& &+\int_{(0,\infty)}\left([V_{z_{1}}^{z_{2}}(x-y)-V_{x}(x-y)]-[V_{z_{1}}^{z_{2}}(x)-V_{x}(x)]\right)\upsilon(\mathrm{d}y),
\end{eqnarray}
where the last equality stems from $[V_{z_{1}}^{z_{2}}]^{\prime}(x)=V_{x}^{\prime}(x)=1$ and $[V_{z_{1}}^{z_{2}}]^{\prime\prime}(x)=0$ for $x>z_{2}$.

Similarly, by (\ref{decreasingsuff.cond.}) and (\ref{barrierstra.Vxy})  we have
\begin{equation}\label{64}
\lim_{y\uparrow x}V_{x}^{\prime\prime}(y)\geq0,\quad x\in(z_{2},\infty)\setminus (d_{i})_{i\leq m}.\nonumber
\end{equation}
By Lemma \ref{4.6}, it holds that $V_{z_{1}}^{z_{2}}(x)-V_{x}(x)\geq0$ and
\begin{eqnarray}\label{66}
\left[V_{z_{1}}^{z_{2}}(x-y)-V_{x}(x-y)\right]&-&\left[V_{z_{1}}^{z_{2}}(x)-V_{x}(x)\right]\nonumber\\
&=&h(x-y)-h(x)\leq0, \quad y\in[0,\infty).\nonumber
\end{eqnarray}
 Therefore, the right-hand side of (\ref{sign.lim.oper.}) is non-positive, and this proves \eqref{ver.xbig.a*}.


Now, as per Lemma \ref{Lem.1},  the $(z_{1},z_{2})$ strategy is optimal among all admissible IDCI strategies. The proof is completed.
\qed
\end{proof}

%



\section{A numerical example}\label{sec:5}
To illustrate the findings in previous sections, we provide some numerical results in this section.

Assume that the driven process follows
$$
X(t)=\mu t+\sigma B(t),\quad t\geq 0,
$$
a Brownian motion with drift,  where $\mu\in \mathbb{R}$, $\sigma>0$, and $\{B(t)\}$ is the standard Brownian motion.
As per \cite{Kyprianou12}, the $q$-scale function for the above Brownian motion is
\begin{eqnarray*}
W^{(q)}(x)&=&\frac{\exp\Big\{\frac{-\mu+\sqrt{ \mu^2+2q\sigma^2 }}{\sigma^2} x\Big\}
-\exp\Big\{\frac{-\mu-\sqrt{ \mu^2+2q\sigma^2 }}{\sigma^2} x\Big\}}{\sqrt{ \mu^2+2q\sigma^2}}\\
&:=&\frac{1}{\sigma^2 \delta}\left(\mathrm{e}^{(-w+\delta) x}-\mathrm{e}^{-(w+\delta) x}\right),\quad x\geq 0,
\end{eqnarray*}
where
$\delta=\frac{\sqrt{ \mu^2+2q\sigma^2 } }{ \sigma^2}$
and
$w=\frac{\mu }{ \sigma^2}$. Let $\alpha=w+\delta$ and $\beta=w-\delta$. By definition we have
\begin{eqnarray*}
Z^{(q)}(x)&=&1+q\int_{0}^{x}W^{(q)}(z)\mathrm{d}z
=\frac{1}{2 \delta}\left(\alpha \mathrm{e}^{-\beta x}-\beta \mathrm{e}^{-\alpha x}\right),\quad x\geq 0,\\
\overline{Z}^{(q)}(x)&=&\int_{0}^{x}Z^{(q)}(z)\mathrm{d}z
=-\frac{\mu}{q}+\frac{\sigma^2 }{4q \delta}\left(\alpha^{2}\mathrm{e}^{-\beta x}-\beta^{2}\mathrm{e}^{-\alpha x}\right),\quad x\geq 0.
\end{eqnarray*}
Hence, for $0<c\leq z_{1}+c< z_{2}<\infty$, it holds that
\begin{eqnarray}\label{explicit.auxiliary.func.}
\hspace{-0.5cm}\xi(z_{1},z_{2})
&=&\frac{2\delta(z_{2}-z_{1}-c)}{\zeta(z_1, z_2)}-\frac{\phi\mu}{q}-\frac{\phi(\mathrm{e}^{-\beta z_{2}}-\mathrm{e}^{-\beta z_{1}}-\mathrm{e}^{-\alpha z_{2}}+\mathrm{e}^{-\alpha z_{1}})}{\zeta(z_1, z_2)},
\end{eqnarray}
where $\zeta(z_1, z_2)=\alpha(\mathrm{e}^{-\beta z_2}-\mathrm{e}^{-\beta z_1})-\beta(\mathrm{e}^{-\alpha z_2}-\mathrm{e}^{-\alpha z_1})$.
Differentiating both sides of (\ref{explicit.auxiliary.func.}) with respect to $z_1$ we get
\begin{eqnarray}\label{explicit.repr.of.deri.xi.z1}
\hspace{-0.6cm}\frac{\partial}{\partial z_1}\xi(z_{1},z_{2})
&=&-\frac{2\delta+\phi(\beta \mathrm{e}^{-\beta z_1}-\alpha \mathrm{e}^{-\alpha z_1})}{\zeta(z_{1},z_{2})}-\frac{\xi(z_1,z_2)+\frac{\phi \mu}{q}}{\zeta(z_{1},z_{2})}\frac{\partial}{\partial z_1}\zeta(z_{1},z_{2}).
\end{eqnarray}
By solving $\frac{\partial}{\partial z_{1}}\xi(z_{1},z_{2})=0$ we get
\begin{equation}\label{xi2.0}
\xi(z_{1},z_{2})=-\frac{2\delta+\phi(\beta \mathrm{e}^{-\beta z_1}-\alpha \mathrm{e}^{-\alpha z_1})}{\alpha\beta(\mathrm{e}^{-\beta z_1}-\mathrm{e}^{-\alpha z_1})}-\frac{\phi\mu}{q}.
\end{equation}
Differentiating both sides of \eqref{explicit.auxiliary.func.} with respect to $z_{2}$ we get
\begin{eqnarray}
\label{explicit.repr.of.deri.xi.z2}
\hspace{-0.5cm}\frac{\partial}{\partial z_2}\xi(z_{1},z_{2})
&=&\frac{2\delta+\phi(\beta \mathrm{e}^{-\beta z_2}-\alpha \mathrm{e}^{-\alpha z_2})}{\zeta(z_{1},z_{2})}-\frac{\xi(z_1,z_2)+\frac{\phi \mu}{q}}{\zeta(z_{1},z_{2})}\frac{\partial}{\partial z_2}\zeta(z_{1},z_{2}).
\end{eqnarray}
Setting $\frac{\partial}{\partial z_{2}}\xi(z_{1},z_{2})=0$ in \eqref{explicit.repr.of.deri.xi.z2} we solve
\begin{eqnarray}\label{xi3.0}
\xi(z_{1},z_{2})=\frac{2\delta+\phi(\beta \mathrm{e}^{-\beta z_2}-\alpha \mathrm{e}^{-\alpha z_2})}{\alpha\beta(\mathrm{e}^{-\alpha z_2}-\mathrm{e}^{-\beta z_2})}-\frac{\phi\mu}{q}.
\end{eqnarray}
By \eqref{explicit.repr.of.deri.xi.z1} one can verify that
$\frac{\partial}{\partial z_{1}}\xi(0,z_{2})=\frac{2\delta(\phi-1)}{\alpha(\mathrm{e}^{-\beta z_{2}}-1)-\beta(\mathrm{e}^{-\alpha z_{2}}-1)}>0$,
excluding the possibility for the maximizer of $\xi$ to lie on the line $z_{1}=0$.
Since it is proved (cf., Proposition \ref{Pro.2}) that the maximizer of $\xi$ cannot be attained on the line $z_{2}=z_{1}+c$, we claim that the $\xi$  is maximized at an interior point of the set $\{(z_{1},z_{2});z_{1},z_{2}\in[0, z_{0}],z_{1}+c\leq z_{2}\}$ for some bounded $z_{0}>0$ (see the arguments immediately following \eqref{Def. set of maximizers}).
Thus, if $(z_{1},z_{2})$ is the maximizer of $\xi$, then \eqref{explicit.auxiliary.func.}, \eqref{xi2.0} and \eqref{xi3.0} should hold simultaneously.
Combining \eqref{xi2.0} and \eqref{xi3.0} yields
\begin{eqnarray}\label{xi2.0+xi3.0}
\mathrm{e}^{-\alpha z_2}-\mathrm{e}^{-\alpha z_1}-\mathrm{e}^{-\beta z_2}+\mathrm{e}^{-\beta z_1}+\phi(\mathrm{e}^{-\beta z_2-\alpha z_1}-\mathrm{e}^{-\alpha z_2-\beta z_1})=0.
\end{eqnarray}
Similarly, combining \eqref{explicit.auxiliary.func.} and \eqref{xi2.0} yields
\begin{eqnarray}\label{xi2.0+xi}
\alpha\beta(z_2-z_1-c)(\mathrm{e}^{-\beta z_1}-\mathrm{e}^{-\alpha z_1})&+&\zeta(z_1, z_2)+2\delta\phi \mathrm{e}^{-\frac{2\mu}{\sigma^2}z_1}\nonumber\\
&-&\alpha\phi \mathrm{e}^{-\alpha z_1-\beta z_2}+\beta\phi \mathrm{e}^{-\alpha z_2-\beta z_1}=0.
\end{eqnarray}

Now, we  are ready to present the numerical results. First, we set $\mu=1, \sigma=0.36, q=0.05$, $c=0.1$ and $\phi=0.05$. Numerically, \eqref{xi2.0+xi3.0} and \eqref{xi2.0+xi} are uniquely solved by $(z_1,z_2)=(0.02682,2.12950)$, a maximizer of $\xi$. According to the previous argument, it must be the maximizer of $\xi$. In fact, by routine calculus we can verify that, at  $(z_1,z_2)=(0.02682,2.12950)$,
\begin{eqnarray*}
\frac{\partial^2\xi(z_1, z_2)}{\partial z_1^2}&=&\frac{\phi(\beta^2 \mathrm{e}^{-\beta z_1}-\alpha^2 \mathrm{e}^{-\alpha z_1})+\frac{2\delta+\phi(\beta \mathrm{e}^{-\beta z_1}-\alpha \mathrm{e}^{-\alpha z_1})}{\mathrm{e}^{-\beta z_1}-\mathrm{e}^{-\alpha z_1}}(\alpha \mathrm{e}^{-\alpha z_1}-\beta \mathrm{e}^{-\beta z_1})}{\zeta(z_1, z_2)}<0,\\
\frac{\partial^2\xi(z_1, z_2)}{\partial z_2^2}&=&\frac{\phi(\alpha^2 \mathrm{e}^{-\alpha z_2}-\beta^2 \mathrm{e}^{-\beta z_2})-\frac{2\delta+\phi(\beta \mathrm{e}^{-\beta z_2}-\alpha \mathrm{e}^{-\alpha z_2})}{\mathrm{e}^{-\alpha z_2}-\mathrm{e}^{-\beta z_2}}(\beta \mathrm{e}^{-\beta z_2}-\alpha \mathrm{e}^{-\alpha z_2})}{\zeta(z_1, z_2)}<0,\\
\frac{\partial^2\xi(z_1, z_2)}{\partial z_1\partial z_2}&=&\frac{\partial^2\xi(z_1, z_2)}{\partial z_2\partial z_1}=0,
\end{eqnarray*}
and hence
$\frac{\partial^2\xi(z_1, z_2)}{\partial z_1^2}\frac{\partial^2\xi(z_1, z_2)}{\partial z_2^2}-\frac{\partial^2\xi(z_1, z_2)}{\partial z_1\partial z_2}\frac{\partial ^2\xi(z_1, z_2)}{\partial z_2\partial z_1}>0$,
 verifying that $(z_1,z_2)=(0.02682,2.12950)$ is the maximizer of $\xi$. This is also confirmed in Figure \ref{Fig1}. Also, as seen in Figure \ref{Fig2},
 $$G(x):=\phi q[W^{(q)}(x)]^2+[1-\phi Z^{(q)}(x)][W^{(q)}(x)]^{\prime}\ge 0,\quad\mbox{for $x\ge z_2=2.12950.$}$$
 This verifies \eqref{decreasingsuff.cond.}.

\begin{figure}[ht]\centering
\subfigure[{Global maximizer of $\xi(z_1, z_2)$}]
{\includegraphics[height=0.15\textheight,width=0.5\textwidth]{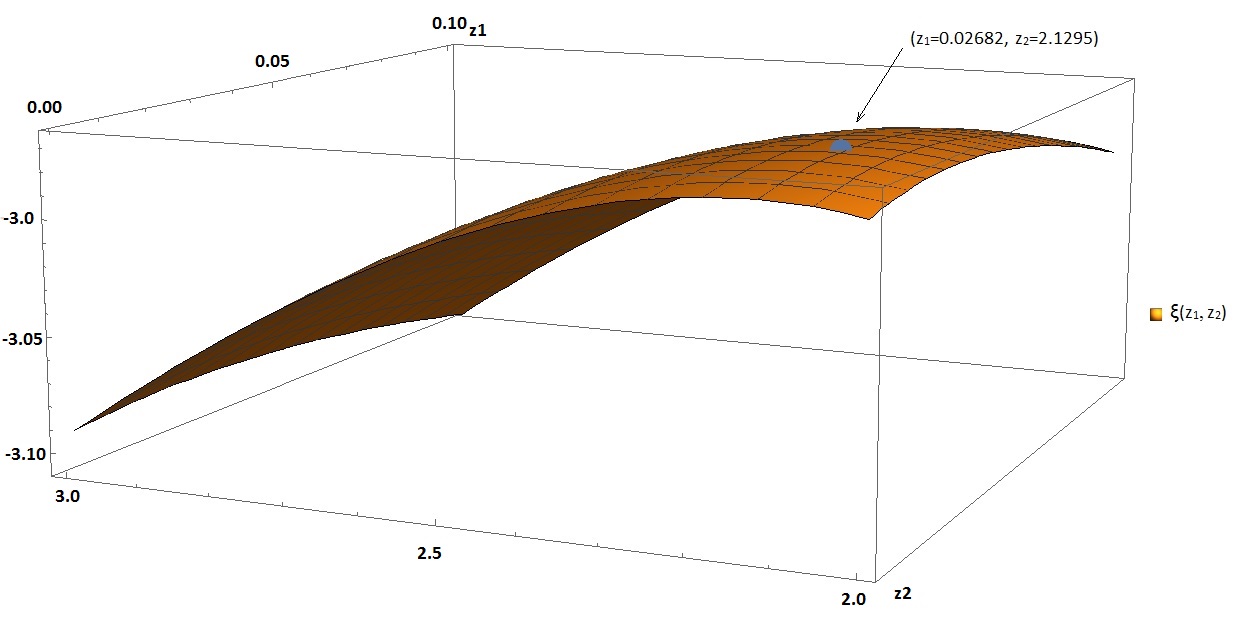}\label{Fig1}}\quad
\subfigure[{Curve of $G(x)$}]
{\includegraphics[height=0.15\textheight,width=0.4\textwidth]{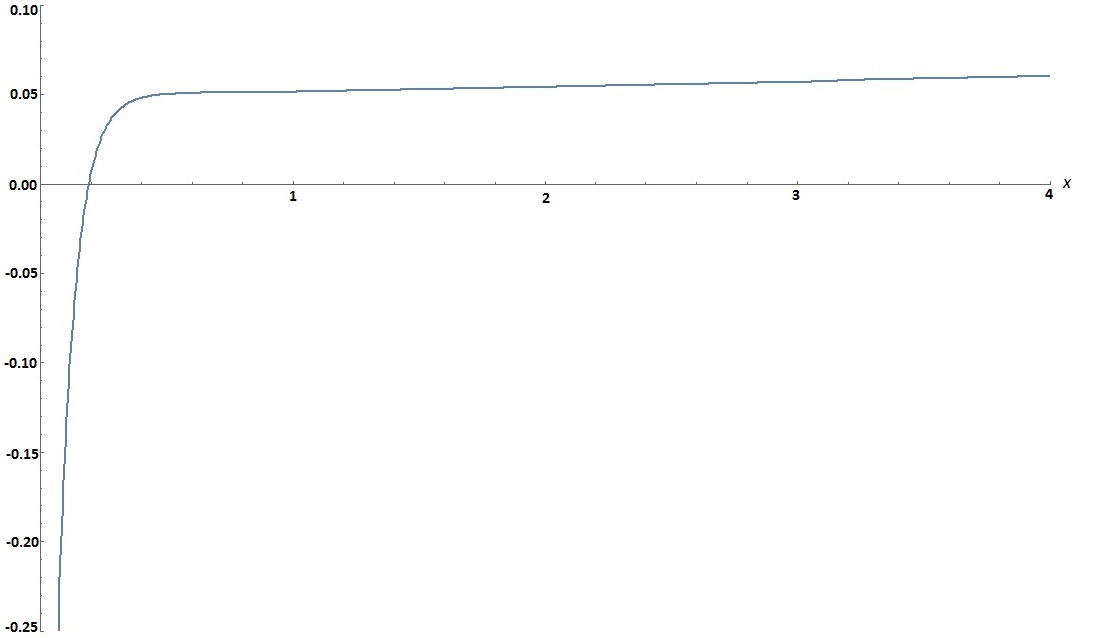}\label{Fig2}}
\caption{Surface of $\xi(z_1, z_2)$ and curve of $G(x)$}
\end{figure}

With the optimal $(z_1,z_2)=(0.02682,2.12950)$ strategy, we can plot its associated value function $V^{z_2}_{z_1}(x)$. According to Proposition 3.2, we have
\begin{eqnarray*}
V^{z_2}_{z_1}(x)=\left\{\begin{array}{l l}
\frac{2\delta(\alpha \mathrm{e}^{-\beta x}-\beta \mathrm{e}^{-\alpha x})}{\xi(0.02682, 2.1295)}+\frac{\phi\sigma^2 }{4q \delta}(\alpha^{2}\mathrm{e}^{-\beta x}-\beta^{2}\mathrm{e}^{-\alpha x}), &0\le x\le 2.1295,\\
x-2.1295+V^{z_2}_{z_1}(2.1295), &x>2.1295.\end{array}\right.
\end{eqnarray*}
It is observed in Figure \ref{Fig3} that the segment in blue (i.e. $x\le 2.1295$) is shaped similar to a straight line, even though its underlying function is actually a combination of exponential functions.

\begin{figure}[ht]\centering
\subfigure[{Curve of $V^{z_2}_{z_1}(x)$}]
{\includegraphics[height=0.15\textheight,width=0.45\textwidth]{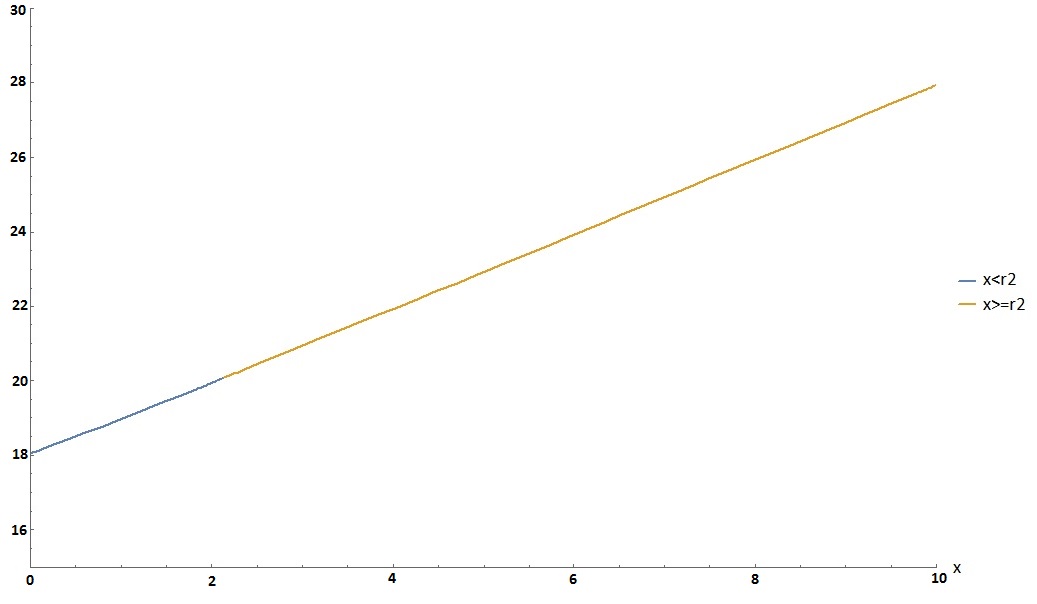}\label{Fig3}}\quad
\subfigure[{Optimal lump sum dividend amount w.r.t. $c$}]
{\includegraphics[height=0.15\textheight,width=0.45\textwidth]{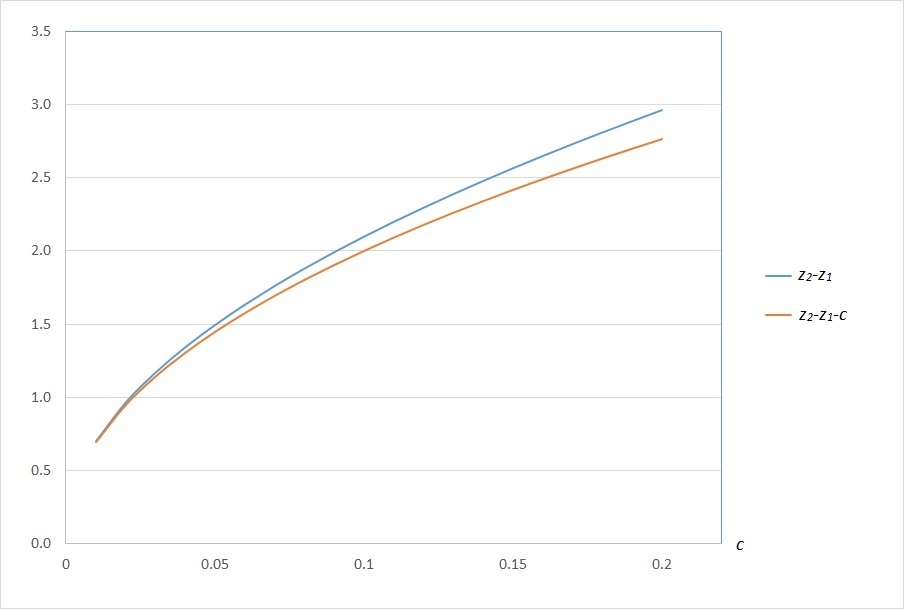}\label{Fig4}}
\caption{Function $V^{z_2}_{z_1}(x)$ and optimal dividends}
\end{figure}

Next, let us examine the parameter sensitivity concerned with $c$ and $\phi$, both playing a critical role in our model. To avoid repetitiveness, we omit the checking arguments of the maximizers of $\xi$. Also,  for ease of comparison, we set $\mu=1, \sigma=0.36$, and $q=0.05$ thereafter. For $\phi=1.05$, in Table \ref{tab1} of the maximizer of $\xi$ for $c= 0.01, 0.02, \ldots, 0.20$, $z_1$ is seen to have a slow but steady downward trend when $c$ increases while $z_2$ has a solid upward trend. Further, in Figure \ref{Fig4}, the individual dividend amount $z_2-z_1$ and the net individual dividend amount$z_2-z_1-c$ both display a solid increasing trend when the transaction cost $c$ increases. This is reasonable because the better way of paying dividends is to pay out more each time with a higher dividend threshold when transaction cost increases.

\begin{table}[ht]
\begin{center}
\begin{tabular}{c| c c| c| c c} \hline
$c$	&$z_1$			&$z_2$				&$c$	&$z_1$		&$z_2$ \\ \hline
0.01	&0.06002	&0.76463		&0.11	&0.02583	&2.22967\\
0.02	&0.04818	&1.01761		&0.12	&0.02496	&2.32560\\
0.03	&0.04195	&1.21568		&0.13	&0.02418	&2.41782\\
0.04	&0.03787	&1.38447		&0.14	&0.02348	&2.50673\\
0.05	&0.03491	&1.53426		&0.15	&0.02284	&2.59269\\
0.06	&0.03262	&1.67044		&0.16	&0.02226	&2.67598\\
0.07	&0.03077	&1.79622		&0.17	&0.02172	&2.75685\\
0.08	&0.02924	&1.91373		&0.18	&0.02123	&2.83550\\
0.09	&0.02794	&2.02447		&0.19	&0.02077	&2.91212\\
0.10	&0.02682	&2.12950		&0.20	&0.02034	&2.98686\\ \hline
\end{tabular}
\end{center}
\caption{Maximizer of $\xi$ with respect to $c$ when $\phi=1.05$}\label{tab1}
\end{table}

\begin{table}[ht]
\begin{center}
\begin{tabular}{c| c c| c| c c} \hline
$\phi$	&$z_1$	&$z_2$		&$\phi$	&$z_1$	&$z_2$\\ \hline
1.01	&0.00635	&2.10904		&1.11	&0.04877	&2.15145\\
1.02	&0.01212	&2.11481		&1.12	&0.05179	&2.15447\\
1.03	&0.01741	&2.12010		&1.13	&0.05467	&2.15736\\
1.04	&0.02229	&2.12497		&1.14	&0.05743	&2.16011\\
1.05	&0.02682	&2.12950		&1.15	&0.06008	&2.16276\\
1.06	&0.03104	&2.13373		&1.16	&0.06262	&2.16530\\
1.07	&0.03501	&2.13769		&1.17	&0.06506	&2.16774\\
1.08	&0.03873	&2.14142		&1.18	&0.06741	&2.17010\\
1.09	&0.04226	&2.14494		&1.19	&0.06968	&2.17236\\
1.10	&0.04559	&2.14828		&1.20	&0.07187	&2.17456\\ \hline
\end{tabular}
\end{center}
\caption{Maximizer of $\xi$ with respect to $\phi$ when $c=0.1$}\label{tab2}
\end{table}

\begin{figure}[ht] \centering
   \includegraphics[height=0.175\textheight,width=0.65\textwidth]{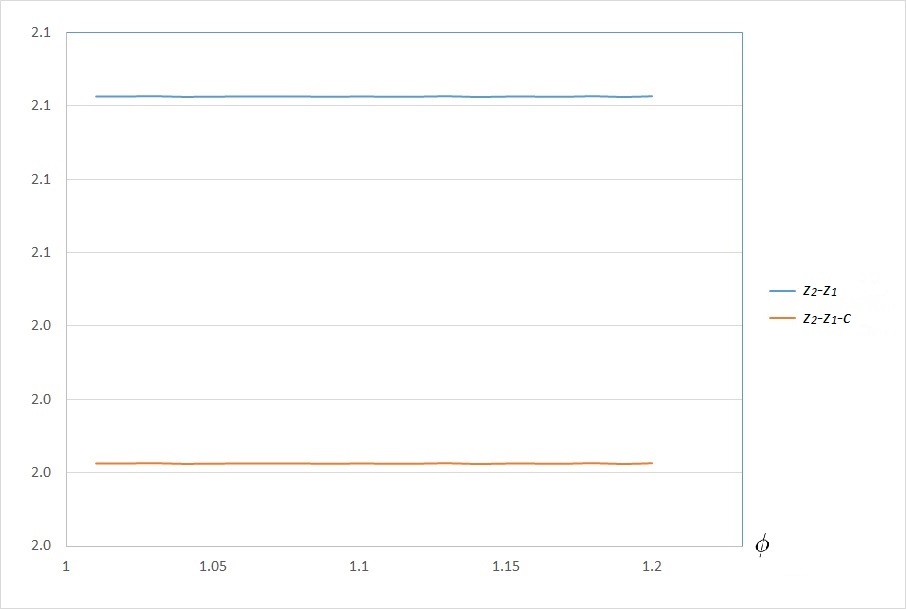}
   \caption{Optimal lump sum dividend amount w.r.t. $\phi$}\label{Fig5}
\end{figure}

For $c=0.1$, Table \ref{tab2} lists the maximizer $\xi$ for $\phi= 1.01, 1.02, \ldots, 1.20$. Both $z_1$ and $z_2$ are seen to have steady upward trends when $\phi$ increases. However, $z_2-z_1$ and $z_2-z_1-c$ in this case almost keep constant no matter how $\phi$ changes. As seen in Figure \ref{Fig5}, when the cost of capital injection goes up, it is more beneficial to have a higher dividend threshold, which partially reduces the chance of needing capital injection. Also, the increasing trend of $z_1$ upon $\phi$ lowers the negative impact of dividends on the solvency of the insurer, while at the same time helping the company to reduce the need of additional capital. On the other hand, the amount of money paid out in each dividend does not depend on $\phi$, but on the value of $c$ which has been observed in the previous case.



\begin{acknowledgements}
The authors are grateful to the anonymous referees for their very careful reading of the paper, and for their very constructive and helpful suggestions and comments.
The authors are also grateful to Professor Xiaohu Li in the Department of Mathematical Sciences at the Stevens Institute of Technology (USA) for helping to polish the English writing of this paper.

\end{acknowledgements}



\end{document}